\documentclass{amsart}
\usepackage{amsmath,amssymb,amsthm, amscd}
\usepackage[margin=1in]{geometry}
\usepackage{xcolor}
\usepackage{pifont,mathrsfs}
\usepackage{array,lastpage}
\usepackage{enumerate,xspace,pifont,shadow}
\usepackage{mathrsfs}
\usepackage[T1]{fontenc}
\usepackage{mathtools}
\usepackage{algorithm}
\usepackage{algorithmicx}
\usepackage{algpseudocode}
\usepackage{caption}
\usepackage{pgf, tikz, tikz-cd}
\usetikzlibrary[quotes]
\usepackage{hyperref}

\DeclareSymbolFont{AMSb}{U}{msb}{m}{n}
\DeclareMathSymbol{\Z}{\mathbin}{AMSb}{"5A}
\DeclareMathSymbol{\R}{\mathbin}{AMSb}{"52}
\DeclareMathSymbol{\N}{\mathbin}{AMSb}{"4E}
\DeclareMathSymbol{\Q}{\mathbin}{AMSb}{"51}
\DeclareMathSymbol{\F}{\mathbin}{AMSb}{"46}
\DeclareMathSymbol{\C}{\mathbin}{AMSb}{"43}

\newcommand{\B}{\mathbf{b}}
\newcommand{\avec}{\mathbf{a}}
\newcommand{\bvec}{\mathbf{b}}
\newcommand{\cvec}{\mathbf{c}}
\newcommand{\dvec}{\mathbf{d}}

\newcommand{\fvec}{\mathbf{f}}
\newcommand{\gvec}{\mathbf{g}}
\newcommand{\hvec}{\mathbf{h}}
\newcommand{\kvec}{\mathbf{k}}

\newcommand{\uvec}{\mathbf{u}}
\newcommand{\vvec}{\mathbf{v}}
\newcommand{\wvec}{\mathbf{w}}
\newcommand{\xvec}{\mathbf{x}}
\newcommand{\yvec}{\mathbf{y}}

\newcommand{\Bset}{\mathcal{B}}

\newcommand{\ev}{\textup{ev}}
\newcommand{\Span}{\textup{Span}}
\newcommand{\EQP}{\textup{EQP}}

\DeclarePairedDelimiter\floor{\lfloor}{\rfloor}
\DeclarePairedDelimiter\nearest{\lfloor}{\rceil}
\providecommand{\abs}[1]{\lvert#1\rvert}
\providecommand{\norm}[1]{\lvert\lvert#1\rvert\rvert}

\newtheorem{thm}{Theorem}[section]
\newtheorem{lem}[thm]{Lemma}
\newtheorem{cor}[thm]{Corollary}
\newtheorem{prop}[thm]{Proposition}

\newtheorem{fact}[thm]{Fact}

\newtheorem{quest}[thm]{Question}

\theoremstyle{definition}
\newtheorem{definition}[thm]{Definition}

\theoremstyle{remark}
\newtheorem{remark}[thm]{Remark}

\theoremstyle{remark}
\newtheorem{example}[thm]{Example}

\theoremstyle{remark}
\newtheorem{claim}[thm]{Claim}

\theoremstyle{remark}

\theoremstyle{remark}

\theoremstyle{remark}

\begin{document}
\bibliographystyle{amsplain}

\title[A parametric version of LLL and some consequences]{A parametric version of LLL and some consequences: parametric shortest and closest vector problems}
\author{Tristram Bogart, John Goodrick, and Kevin Woods}

\begin{abstract}
Given a parametric lattice with a basis  given by polynomials in $\Z[t]$, we give an algorithm to construct an LLL-reduced basis whose elements are eventually quasi-polynomial in $t$: that is, they are given by formulas that are piecewise polynomial in $t$ (for sufficiently large $t$), such that each piece is given by a congruence class modulo a period. As a consequence, we show that there are parametric solutions of the shortest vector problem (SVP) and closest vector problem (CVP) that are also eventually quasi-polynomial in $t$.
\end{abstract}

\maketitle

\section{Introduction} \label{sec:intro}

This paper addresses the following three related questions:

\begin{quest}
\label{main_quest}
Suppose that $\fvec_1(t), \ldots, \fvec_n(t) \in \Z[t]^m$ are $n$ integer vectors in $\Z^m$ whose coordinates vary according to polynomials in $\Z[t]$, where $t \in \N$. Let $\Lambda_t$ be the lattice spanned by $\fvec_1(t), ..., \fvec_n(t)$. \footnote{Note that we do \emph{not} assume that $\fvec_1(t), ..., \fvec_n(t)$ are linearly independent, that is, they may not be a basis for $\Lambda_t$.}

\begin{enumerate}
\item What type of function $\uvec : \N \rightarrow \Z^m$ can we obtain that selects a shortest nonzero vector in the lattice $\Lambda_t$? 
\item Suppose $\xvec(t) \in \Q(t)^m$ is another parametric vector. What type of function $\uvec : \N \rightarrow \Z^m$ can we obtain that selects a vector in $\Lambda_t$ which is as close as possible to $\xvec(t)$?
\item What type of functions $\gvec_1, \ldots, \gvec_{n'}:\N \rightarrow \Z^m$ can we obtain so that $\mathcal{B}_t = \{\gvec_1(t), \ldots, \gvec_{n'}(t)\}$ is an LLL-reduced basis (to be defined in Section 2) of $\Lambda_t$?
\end{enumerate}

\end{quest}

If the vectors $\fvec_1, \ldots, \fvec_n$ and $\xvec$ do not depend on $t$, then finding solutions $\uvec$ in (1) and (2) are known as the Shortest Vector Problem (SVP) and Closest Vector Problem (CVP), respectively, and computing an LLL-reduced basis (3) is one common technique to tackle these and many other problems; see, for example, Galbraith's text \cite[Chapters 17 and 18]{Galbraith}.

To get a feel for how solutions $\uvec(t)$ to the \emph{parametric} SVP might look, let us consider a few examples. If $\fvec_1(t) = (t,2)$ and $\fvec_2(t) = (1,t^2)$, then a shortest nonzero vector in $\Lambda_t$ is $\uvec(t)=(t,2)$ for all $t \geq 3$ but something else for $t=0,1,2$. That is, a  formula for $\uvec(t)$ might only stabilize \emph{eventually}, after finitely many exceptional values. On the other hand, if $\fvec_1(t) = (3,0)$ and $\fvec_2(t) = (2t,1)$, then a shortest vector in $\Lambda_t$ is
\[\uvec(t)=\begin{cases}(0,1)&\text{if $t\equiv 0\pmod 3$},\\
(-1,1)&\text{if $t\equiv 1\pmod 3$},\\
(1,1)&\text{if $t\equiv 2\pmod 3$}.
\end{cases}\]
That is, a  formula  for $\uvec(t)$ might depend on the congruence class of $t$ modulo a period. 

In this paper we give an answer to Question~\ref{main_quest}, showing that examples such as the two above yield the only complications in formulas for $\uvec(t)$. To be precise, our solution to the parametric SVP is:

\begin{thm} \label{thm:shortest}
Let $\Lambda_t = \Span_\Z \{\fvec_1(t), \dots, \fvec_n(t)\} \subseteq \Z[t]^m$. Then for all sufficiently large values of $t \in \N$, a shortest nonzero vector $\uvec(t) \in \Lambda_t$ can be produced by a formula
\begin{equation*}
\uvec(t) = \gvec_i(t), \textup{ if } t \equiv i \textup{ (mod } N \textup{)}
\end{equation*}
for fixed $N \in \N$ and $\gvec_0, \ldots, \gvec_{N-1} \in \Q[t]^m$.
\end{thm}

A vector-valued function $\uvec: \N \rightarrow \Z^m$ of the form given in Theorem~\ref{thm:shortest} is called \emph{eventually quasi-polynomial} or ``EQP'' for short. (It is \emph{quasi-polynomial} if the formula holds for all $t \in \N$, not just for $t$ sufficiently large.)

The answer to part (2) of the question is the same:

\begin{thm} \label{thm:closest}
Let $\Lambda_t = \Span_\Z \{\fvec_1(t), \dots, \fvec_n(t)\} \subseteq \Z[t]^m$, and let $\xvec(t)\in \Q(t)^m$. Then  there exists $\uvec:\N\rightarrow\Z^m$ with EQP coordinates such that $\uvec(t)$ is a vector in $\Lambda_t$ that is closest to $\xvec(t)$. 
\end{thm}

In proving both Theorem~\ref{thm:shortest} and Theorem~\ref{thm:closest}, we adapt the classical LLL reduction algorithm of Lenstra, Lov\'asz, and Lov\'asz \cite{LLL} to the parametric case. In particular, we obtain the following result, which is of independent interest:

\begin{thm} \label{thm:LLLreduced}
Let $\Lambda_t = \Span_\Z \{\fvec_1(t), \dots, \fvec_n(t)\} \subseteq \Z[t]^m$. Then there exists a set $\mathcal{B}_t = \{\gvec_1(t), \ldots, \gvec_{n'}(t)\}$ of EQP vectors such that for all sufficiently large $t$, we have that $\mathcal{B}_t$ is an LLL-reduced basis for $\Lambda_t$.\footnote{Note that we do not assume that the $\fvec_i$ are linearly independent, so possibly $n' < n$.}

\end{thm}

Finally, we note that in our main results above, our solutions only work \emph{eventually} (for all sufficiently large values of $t$). One might ask for effective bounds on the parameter controlling at which point our ``eventual'' solutions become correct (say, in terms of the degrees of the input polynomials and the sizes of their coefficients), but unfortunately it is not clear (or at least to us) how to obtain such bounds based on the techniques of this paper. So we leave this as a question for further research.

\subsection{Context and methods}

The SVP has a long history going back at least to Gauss \cite{gauss}, who gave a formula for $n=2$ dimensions. If we only need a solution to SVP which is \emph{approximately} correct within a factor of $2^{n/2}$, then the well-known polynomial-time ``LLL-algorithm'' of Lenstra, Lov\'asz, and Lov\'asz works \cite{LLL}; we will have much more to say about this below. Van Emde Boas showed that both SVP with respect to the $L_\infty$-norm and CVP with respect to the Euclidean norm are NP-hard \cite{vEB}, and conjectured that SVP with $L_2$ is NP-hard, although to our knowledge this question remains open today.  Henk \cite{Henk} showed that SVP is polynomial-time reducible to CVP, and Ajtai has proved that SVP is NP-hard under randomized reductions \cite{Ajtai}. Micciancio \cite{micciancio} has a more recent discussion of the somewhat delicate complexity issues here. It is also notable that Ajtai and Dwork \cite{AjtaiDwork} have developed a public-key cryptography system which is provably secure if in the worst case \emph{unique-SVP} (a closely related problem) has no probabilistic polynomial-time solution.

Many previous papers have noted that EQP functions are useful for counting and specification problems related to integer programming and properties of lattices. For example, Calegari and Walker  \cite{CW} gave EQP formulas for the vertices of integer hulls of certain families of polytopes presented as convex hulls of vectors parameterized by rational functions; Chen, Li, and Sam \cite{CLS} showed that the number of integer points in a polytope whose bounding hyperplanes are given by rational functions is EQP; and Shen \cite{Shen15b} has shown that solutions to integer programming problems parameterized by polynomials can be given by EQP formulas.

The hope that parametric shortest and closest lattice vector problems might also have EQP solutions came from considering one of our own previous results from \cite{BGW2017}. To explain this, we need the following generalization of the notion of a \emph{Presburger set} from mathematical logic:

\begin{definition} (from Woods \cite{Woods2014})
A \emph{parametric Presburger family} is a collection $\{S_t : t \in \N\}$ of subsets of $\Z^m$ which can be defined by a formula using addition, inequalities, multiplication and addition by constants from $\Z$, Boolean operations (and, or, not), multiplication by $t$, and quantifiers ($\forall$, $\exists$) on variables ranging over $\Z$. That is, such families are defined using quantifiers and Boolean combinations of formulas of the form $\avec(t)\cdot \xvec \le b(t),$ where $\avec (t)\in\Z[t]^m, b(t)\in\Z[t]$.
\end{definition}

\begin{thm} \cite{BGW2017} \label{thm:BGW}
  Let $\{S_t\}$ be a parametric Presburger family. Then:
  \begin{enumerate}
  \item There exists an EQP $g:\N\rightarrow\N$ such that, if $S_t$ has finite cardinality, then $g(t)=\abs{S_t}$. The set of $t$ such that $S_t$ has finite cardinality is eventually periodic.
  \item There exists a function $\xvec:\N\rightarrow\Z^m$, whose coordinate functions are EQPs, such that, if $S_t$ is non\-emp\-ty, then $\xvec(t)\in S_t$. The set of $t$ such that $S_t$ is nonempty is eventually periodic.
  \end{enumerate}
\end{thm}

\begin{cor} \label{cor:IP}
Suppose $S_t\subseteq \Z^m$ is a parametric Presburger family. Given $\cvec \in \Z^m\setminus\{0\}$, there exists a function $\xvec:\N \rightarrow \Z^m$ such that, if $\min_{\yvec\in S_t}  \cvec \cdot \yvec$ exists, then it is attained at $\xvec(t)\in S_t$, and the coordinate functions of $\xvec$ are EQPs. The set of  $t$ such that the maximum exists is eventually periodic.
\end{cor}

Note that  $\Lambda_t = \Span_\Z \{\fvec_1(t), \dots, \fvec_n(t)\} $ is a parametric Presburger family, as is
\[S_t=\Big\{(y,\xvec):\ \xvec\in\Lambda_t\ \wedge\ \xvec\ne  \mathbf 0\ \wedge\ ||\xvec||_1\le y\Big\},\]
where $||\xvec||_1$ is the $L_1$-norm, $\sum_{i=1}^m | x_i |$. Therefore, if we wanted to minimize the $L_1$-norm over the nonzero vectors in  a parametric lattice, we could apply Corollary \ref{cor:IP} to $S_t$ with $c=(1,\mathbf 0)$, and we would already be guaranteed that the minimum and the argmin are specified by EQP functions. More generally, if $\Lambda_t$ is a parametric lattice (or in fact any Presburger family), then minimizing any \emph{polyhedral norm} will yield an EQP, by the same reasoning. A polyhedral norm is defined by taking a full dimensional, centrally symmetric polytope $P$, and defining $||\xvec|| = \min\{t:\ \xvec \in tP\}$. This includes $L_1$ ($P$ is the cross-polytope) and $L_{\infty}$ ($P$ is the cube).

However, it is not clear (to us) how Theorem~\ref{thm:BGW} could be applied to finding the shortest vector in the \emph{Euclidean} norm: parametric Presburger families require formulas that are \emph{linear} in the variables $\xvec$. To the best of our knowledge, this present paper is the first result in this area where a \emph{nonlinear} setup still yields EQP results.

Our solution to Question~\ref{main_quest} is based on parametric generalizations of standard techniques for analyzing integer lattices such as LLL reduction and Babai's \cite{Babai86} nearest plane method.

\subsection{Outline of paper}

This paper is organized as follows. First, in Section 2 we recall the original (non-parametric) LLL algorithm for finding lattice bases which are ``almost orthogonal'' and whose $i$-th element approximates the $i$-th longest vector (up to a constant factor). In Section 3, we prove that there is a parametric version of the LLL algorithm (Theorem~\ref{thm:LLLreduced}), which constitutes the bulk of the paper. In Section 4 we quickly deduce from this that the shortest nonzero vector of a parametric lattice is EQP (Theorem~\ref{thm:shortest}) and we adapt Babai's nearest plane method to find an EQP formula for a closest lattice vector (Theorem~\ref{thm:closest}).


\section{Review of LLL-reduced bases}

In this section, we will briefly review the ``classical'' (non-parametric) LLL basis reduction algorithm, introducing some concepts and notation which will be used later in Section 3. Most of this section may be skipped by the reader who already has expertise in this area, though we do establish some important notational conventions (such as $\fvec_1^*, \ldots, \fvec_n^*$ for the \emph{non-normalized} Gram-Schmidt vectors corresponding to the ordered lattice basis $\fvec_1, \ldots, \fvec_n$)

To motivate LLL reduction, first observe that if we have an ordered lattice basis $\fvec_1, \ldots, \fvec_n$ for $\Lambda \subseteq \Z^m$ which happens to be both orthogonal and ordered by length, then  the shortest vector problem is trivial (select $\fvec_1$) and the closest vector problem is easy (project the vector onto the $\R$-span of $\fvec_1, \ldots, \fvec_n$, write the new vector in this basis, and then round off coordinates to the nearest integer). Of course a lattice may have no orthogonal basis, and one way to explain the usefulness of LLL-reduced lattices is that they are \emph{approximately orthogonal} and are \emph{approximately ordered by length} (in a sense to be made more precise below).

Our treatment of the LLL algorithm is based on Galbraith's text \cite[Chapter 17]{Galbraith}. We begin with Gram-Schmidt orthogonalization (Algorithm \ref{alg:GS}), reproduced directly from \cite[Algorithm 24]{Galbraith} except that we replace $\R$ by an arbitrary field $\F$ which is \emph{formally real}: that is, $\F$ can be made into an ordered field, or equivalently a sum of squares of elements of $\F$ is $0$ only if each of those elements is $0$ (a hypothesis which guarantees we will not divide by $0$ at Step 5).

\begin{algorithm} \caption{Gram-Schmidt algorithm without normalization} \label{alg:GS} 
\begin{algorithmic}[1] 
  \Statex \textbf{INPUT:} $\{\bvec_1, \dots, \bvec_n\}$ in $\F^m$
  \Statex \textbf{OUTPUT:} $\{\bvec_1^\ast, \dots, \bvec_n^\ast\}$ in $\F^m$
  \State $\bvec_1^\ast = \bvec_1$
  \For{$i = 2$ to $n$}
    \State $v = \bvec_i$
    \For{$j = i-1$ down to 1}
      \State $\mu_{i,j} = \langle \bvec_i, \bvec_j^\ast \rangle / \langle \bvec_j^\ast, \bvec_j^\ast \rangle$
      \State $v = v - \mu_{i,j} \bvec_j^\ast$
    \EndFor
    \State $\bvec_i^\ast = v$
  \EndFor
  \State \textbf{return} $\{\bvec_1^\ast, \dots, \bvec_n^\ast\}$
\end{algorithmic}
\end{algorithm}

The two formally real fields over which we will apply the algorithm are $\Q$ and $\Q(t)$.


\begin{remark} It is important that we do not use a variant of the Gram-Schmidt algorithm that normalizes the output vectors. Normalization would require dividing by square roots of arbitrary elements of the field $\F$. Neither $\Q$ nor $\Q(t)$ contains such square roots. 
\end{remark}

Since we cannot apply Gram-Schmidt over a ring such as $\Z$ that is not a field, we can only ask for a lattice basis that is approximately orthogonal. For applications it is also useful for the basis to consist of short vectors and to be efficiently computable, for example in the following sense. 

\begin{definition}
  Let $B = (\bvec_1, \dots, \bvec_n)$ be an ordered basis of a lattice $L \subseteq \Z^m$ and $B^\ast = (\bvec_1^\ast, \dots, \bvec_n^\ast)$ be the list of vectors in $\Q^m$ obtained by applying the Gram-Schmidt algorithm without normalization to $(\bvec_1, \dots, \bvec_n)$. For $1 \leq j < i \leq n$, let $\mu_{i,j} = \frac{\langle \bvec_i, \bvec_j^\ast \rangle}{\langle \bvec_j^\ast, \bvec_j^\ast \rangle}$ be the coefficients defined in the Gram-Schmidt process.

  Let $1/4 < \delta < 1$. The basis is \emph{LLL-reduced} (with factor $\delta$) if the following conditions hold:
  \begin{itemize}
  \item (Size reduced)  $|\mu_{i,j}| \leq 1/2$ for $1 \leq j < i \leq n$.  
  \item (Lov\'asz condition) $\| \bvec^*_i \|^2 \geq \left( \delta - \mu_{i,i-1}^2 \right) \| {\bvec^*_{i-1}} \|^2$ for $2 \leq i \leq n$.
  \end{itemize}
\end{definition}

To motivate our definition of LLL-reduced in the parametric context, we note that the Lov\'asz condition can be rewritten as $$\frac{\| \bvec^*_i \|^2}{\| {\bvec^*_{i-1}} \|^2} + \mu_{i,i-1}^2 \geq \delta$$ for $2 \leq i \leq n.$ This version of the Lov\'asz condition has a nice geometric interpretation: an elementary calculation shows that the left-hand side of the inequality above is equal to $\frac{\| \cvec_i \|^2 }{ \| \bvec_{i-1}^* \|^2}$, where $\cvec_i$ is the projection of $\bvec_i$ onto the plane spanned by $\bvec_{i-1}^*$ and $\bvec_i^*$.

Why are LLL-reduced bases nice? Several reasons are given in \cite[Chapter 17]{Galbraith}:

\begin{fact}
\label{fact:Galbr}
(\cite{Galbraith}, Theorem 17.2.12) Suppose that $B = (\bvec_1, \dots, \bvec_n)$ is an LLL-reduced basis for a lattice $\Lambda$ with $\delta = 3/4$. We recursively select vectors $\wvec_1, \ldots, \wvec_n \in \Lambda$ and quantities $\lambda_1 \leq \lambda_2 \leq \ldots \leq \lambda_n$(known as the ``successive minima'') such that for each $i \in \{1, \ldots, n\}$,

\begin{itemize}
\item $\lambda_i = \max \{ \| \wvec_1 \|, \ldots, \| \wvec_i \| \}$, and
\item if $i < n$ then $\wvec_{i+1}$ is a shortest vector in $\Lambda$ which is linearly independent from $\{\wvec_1, \ldots, \wvec_i\}$.
\end{itemize}
(So $\lambda_1$ is the length of a shortest nonzero vector in $\Lambda$.)

Then:

\begin{itemize}
\item $\| \bvec_1 \| \leq 2^{(n-1)/2} \lambda_1$;
\item $2^{(1-i)/2} \lambda_i \leq \| \bvec_i \| \leq 2^{(n-1)/2} \lambda_i$; and
\item $\det (\Lambda) \leq \prod_{i=1}^n \| \bvec_i \| \leq 2^{ {n (n-1)}/4} \det( \Lambda).$
\end{itemize}

\end{fact}

The last condition in the fact above can be thought of as showing that the basis is ``almost orthogonal,'' as the volume of the lattice is close to the product of the lengths of the basis elements (within a constant factor).

For our solution to the parametric Shortest Vector Problem below, the following result from Barvinok's text \cite{Barv08} will be useful:

\begin{fact}
\label{fact:coefbounds}
(\cite{Barv08}, Lemma 12.7) Suppose that $(\bvec_1, \ldots, \bvec_n)$ is an LLL-reduced basis for the lattice $\Lambda$ with $\delta = 3/4$, and $\lambda_1$ is the length of a shortest nonzero vector in $\Lambda$. Then for any $\uvec \in \Lambda$ such that $\| \uvec \| \leq \beta \lambda_1$ (for any constant $\beta \geq 1$), if $$\uvec = \sum_{i=1}^n m_i \bvec_i,$$ then $|m_i| \leq 3^n \beta$ for any $i \in \{1, \ldots, n\}$.

In particular, letting $\beta = 1$, we see that a shortest nonzero vector in $\Lambda$ must have a representation with coefficients $|m_i | \leq 3^n$.
\end{fact}

For the Closest Vector Problem, Babai's ``nearest plane method'' \cite{Babai86} also gives an answer in the unparameterized case which is correct up to a constant factor:

\begin{fact}
\label{fact:babai}
\cite[Theorem 3.1]{Babai86}
If $(\bvec_1, \ldots, \bvec_n)$ is an LLL-reduced basis for a lattice $\Lambda \subseteq \Z^n$ with $\delta = 3/4$, then for any $\xvec \in \R^n$ there is a lattice point $\wvec \in \Lambda$ such that $$\| \xvec - \wvec \| \leq 2^{n/2 - 1} \| \bvec_n^* \|.$$
\end{fact}

Given these nice properties of LLL-reduced bases, it is important to know that they always exist:

\begin{fact} 
\label{thm:classicLLL}
(Lenstra, Lenstra, and Lov\'asz \cite{LLL}) Every sublattice $\Lambda \subseteq \Z^m$ has an LLL-reduced basis.  
\end{fact}

Since all lattice bases are equivalent up to unimodular transformation, we can state and apply the following explicit consequence of Fact \ref{thm:classicLLL}:
\begin{cor}
\label{cor:constructiveLLL}
If $(\avec_1, \dots, \avec_n)$ is any basis of $\Lambda$, then there exists a unimodular matrix $U \in \Z^{m \times m}$ such that $(\bvec_1 \dots \bvec_n) = U (\avec_1 \dots \avec_n)$ is LLL-reduced. 
\end{cor}

The algorithm for reducing parametric lattices will make use of an adaptation of the ``classic'' LLL reduction algorithm of Lenstra, Lenstra, and Lov\'asz, so we will briefly review it  here (Algorithm \ref{alg:LLL}) and sketch the proof of Fact~\ref{thm:classicLLL}. All of the ideas are from \cite{LLL}, but we will follow the exposition of \cite[Algorithm 25]{Galbraith} closely, except that we assume that the input vectors have integer coordinates. This implies that at every step of the algorithm, the lattice basis vectors $\bvec_i$ (in particular, at the end when they form an LLL basis) are integer and the Gram-Schmidt vectors and the coefficients $\mu_{i,j}$ are rational.

\begin{algorithm}  \caption{LLL algorithm with Euclidean norm (typically, choose $\delta = 3/4$)} 
\label{alg:LLL}
\begin{algorithmic}[1]
\Statex \textbf{INPUT:} $\{\B_1, \dots, \B_n\}$ in $\Z^m$, linearly independent over $\Q$
\Statex \textbf{OUTPUT:} An LLL-reduced basis $\{\B_1, \dots, \B_n\}$ in $\Z^m$
\State Compute the Gram-Schmidt basis $\B_1^\ast, \dots, \B_n^\ast$ and coefficents $\mu_{i,j}$ for $1 \leq j < i \leq n$
\State Compute
$B_i = \langle \B_i^\ast, \B_i^\ast \rangle = \| \B_i^\ast \|^2$
for $1 \leq i \leq n$
\State $k=2$
\While{$k \leq n$}
  \For{$j = (k-1)$ down to 1 } \Comment{Perform size reduction}
    \State Let $q_j = \lfloor \mu_{k,j} \rceil$ and set $\B_k = \B_k - q_j\B_j$ \Comment{$\lfloor x \rceil$ denotes closest integer to $x$}
    \State Update the values $\mu_{k,j}$ for $1 \leq j \leq k$
  \EndFor
  \If{$B_k \geq \left(\delta - \mu_{k,k-1}^2\right) B_{k-1}$ } \Comment{Check Lov\'asz condition}
    \State $k = k+1$
  \Else
    \State Swap $\B_k$ with $\B_{k-1}$
    \State Update the values $\B_k^\ast$, $\B_{k-1}^\ast$, $B_k$, $B_{k-1}$, $\mu_{k-1,j}$ and $\mu_{k,j}$ for $1 \leq j < k$, and $\mu_{i,k}, \mu_{i,k-1}$ for $k < i \leq n$
    \State $k = \max\{2, k-1\}$
  \EndIf
\EndWhile  
\end{algorithmic}
\end{algorithm}

It is not too difficult to check that Algorithm~\ref{alg:LLL}, \emph{if} it halts, will output an LLL-reduced basis for the lattice generated by the input basis. For future reference, we recall the argument that Algorithm~\ref{alg:LLL} always terminates after finitely many steps. At any given stage while carrying out the algorithm, we have a lattice basis $\{\B_1, \dots, \B_n\}$, and suppose that we arrive at Line 12 and must swap $\B_k$ and $\B_{k-1}$. Define $B_{(i)}$ to be the $i \times m$ matrix whose rows are given by $\B_1, \ldots, \B_i$ and let $d_i = \det(B_{(i)} B^T_{(i)})$. Since $d_i$ is the square of the volume of the sublattice generated by $\{\B_1, \ldots, \B_i\}$, we have that $d_i$ is an integer and $$d_i = \prod_{j=1}^i \| \B^*_i \|^2.$$ We define $$D = \prod_{i=1}^{n-1} d_i = \prod_{i=1}^{n-1} B_i^{n-i}.$$ An elementary calculation shows that after swapping $\B_k$ and $\B_{k-1}$ and updating all the values $d_i$ and $D$ accordingly, the new value $D'$ satisfies the inequality $0 < D' \leq \delta D$. But at any stage of the algorithm, the quantity $D$ must be a positive integer, so this can only happen finitely often.

\subsection{Issues with the LLL algorithm applied to parametric vectors}
If we wish to apply the LLL algorithm to input vectors $\fvec_1(t), \ldots, \fvec_n(t) \in \Z[t]^m$, then we must begin with Gram-Schmidt reduction over the field $\Q(t) = \textup{Frac} \Z[t]$. We have already seen that this can be done. Next, on line 6 we must make sense of the quantity $q_j = \nearest{\mu_{k,j}}$ for $\mu_{k,j} \in \Q(t)$.

In the lemma below, we note that the rounding off of rational functions can be performed in a way which is EQP. 
\begin{definition} \label{def:mildEQP}
A function $g: \N \rightarrow \Z$ is \emph{mild EQP} if there are polynomials $f_0, \ldots, f_{N-1} \in \Q[t]$ with the same degree and leading coefficient such that eventually $g(t) = f_i(t)$ if $t \equiv i \textup{ mod } N$.
\end{definition}

\begin{lem} \label{lem:closestinteger}
  Let $f, h \in \Z[t]$. Then as $t$ ranges over $\Z$ or $\N$, the floor function $\lfloor \frac{f(t)}{h(t)} \rfloor$ and the closest integer function $\lfloor \frac{p(t)}{q(t)} \rceil$ are mild EQP.  
\end{lem}

\begin{proof}
  We have $\lfloor \frac{f(t)}{h(t)} \rfloor \in \EQP$ by \cite[Theorem 4.1]{CLS}. Although it is not stated explicitly, mildness follows from the proof of this result. 
  Now $\lfloor x \rceil = \lfloor x + \frac{1}{2} \rfloor$ for all $x \in \R$, so $\lfloor \frac{f(t)}{h(t)} \rceil = \lfloor \frac{2f(t) + h(t)}{2h(t)} \rfloor$ is also mild EQP.  
\end{proof}

Now since $q_j$ is EQP, so are the coordinates of the new vector obtained on line 6 and so is its norm. We now run into trouble with the comparison of the norms of two vectors on line 9. If the two norms were simply polynomials or rational functions in $t$, then one side would be eventually larger; i.e., larger when evaluated at any sufficiently large $t$. (See Proposition \ref{prop:polycomparison} below.) However this is not true for quasi-polynomials, so it may not be clear whether or not we should swap the two vectors on line 12.

The natural solution, which we will indeed adopt in our revised algorithm, is to \emph{branch} the computation every time an EQP function is introduced. That is, we assume from then on that $t$ is sufficiently large and that $t \equiv i \, \, (\textup{mod} \, \, M)$ for some fixed $i$ and suitable $M$ so that, after a linear substitution, all of our vectors are over $\Z[t]$ once again.

However, if we continue in this fashion with Algorithm \ref{alg:LLL}, we run into a more serious problem: it is not clear how to prove that it always terminates in finitely many steps. Imitating the termination proof from \cite{Galbraith} (as summarized in the previous subsection), we could define parametric versions $d_i(t)$ and $D(t)$ of the quantities $d_i$ and $D$, which are integer-valued functions; and any time we perform the swap on Line 12, the updated value $D'(t)$ of $D(t)$ will satisfy the same inequality $$0 < D'(t) \leq \delta D(t)$$ for all sufficiently large values of $t \in \N$ (by an identical argument as before). This almost works, except that every time branching occurs, we are effectively performing a substitution of $Mt' + i$ for $t$ where $M,i$ are constants and rewriting all functions in terms of $t'$, which increases the rate of growth of the function $D(t)$.

While a careful analysis might be able to remedy this argument, our approach will be somewhat different, as explained in the following section.

\section{Parametric LLL-reduced bases} \label{sec:LLL-reduced}

In this section we will develop the tools needed to prove Theorem~\ref{thm:LLLreduced}, our parametric version of the LLL algorithm.

Here is a brief summary of the proof. First we note that if $\fvec_1, \ldots, \fvec_n$ are parametric vectors in $\Z[t]^m$ (we use $\fvec_i$ rather than $\bvec_i$ to remind ourselves that they are now functions of $t$) which all have the same degree, and such that their corresponding vectors of leading coefficients are linearly independent, then we can apply the classical LLL algorithm to these leading-coefficient vectors and easily ``lift'' to an asymptotically LLL-reduced basis of the parametric lattice (Proposition~\ref{prop:samedegree}). If $\fvec_1, \ldots, \fvec_n$ have the same degree but their leading coefficient vectors are \emph{not} linearly independent, then certain integer linear combinations will have lower degree, and this case is subsumed in the case of lattices generated by vectors of varying degrees. To deal with the mixed-degree case, we show that the lattice has a basis $\mathcal{B} = \mathcal{B}_0 \cup \ldots \cup \mathcal{B}_d$ where $\mathcal{B}_i$ is a set of degree-$i$ vectors with linearly independent leading coefficients and such that when $i \neq j$, $\mathcal{B}_i$ generates a lattice which is ``asymptotically'' orthogonal to the lattice generated by $\mathcal{B}_j$. Then we can apply classical LLL separately to each $\mathcal{B}_i$ as before. This will ensure the Lov\'asz condition holds, but a further reduction of vectors in each $\mathcal{B}_i$ by basis vectors of lower degree may be necessary to ensure the eventual size-reduced condition holds. For the precise statement of our algorithm, see Section \ref{subsec:algorithm}.

\subsection{Basic definitions}

Here we adapt the various definitions involved in LLL-reduction to the case of parametric lattices.

\begin{definition} $ $
  \begin{enumerate}
  \item A \emph{parametric vector} is a vector $\fvec = (f^{(1)}(t), \dots, f^{(m)}(t))$ whose the coordinate functions vary with $t \in \N$ and such that $\fvec(t) \in \Q^m$ for all $t$. We will say that $\fvec$ is \emph{over $R$} if every $f^{(i)}(t)$ is in the ring of functions $R$. Unless otherwise specificed, parametric vectors will be over $\Z[t]$. 

    \item A set of parametric vectors $\{\fvec_1, \dots, \fvec_n \}$ with values in $\Z^m$ is a \emph{parametric lattice basis} if for every sufficiently large $t \in \N$, the vectors $\{\fvec_1(t), \dots, \fvec_n(t)\}$ are linearly independent and thus span an $n$-dimensional lattice in $\Z^m$. For each $t \in \N $,  $ \Span_\Z \{ \fvec_1(t), \dots, \fvec_n(t) \}$ is a sublattice $\Lambda_t$ of $\Z^m$, and we call $\{\Lambda_t \, : \, t \in \N\}$ the \emph{parametric lattice} generated by $\{\fvec_1, \dots, \fvec_n\}$. We also write $\Span \{\fvec_1, \ldots, \fvec_n\}$ for $\Lambda_t$.

 \item Given a parametric lattice $\Lambda_t$, a parametric vector $\fvec(t)$ is \emph{in $\Lambda_t$} if for all sufficiently large values of $t$, $\fvec(t) \in \Lambda_t$.
\end{enumerate}
  \end{definition}

Note that the dimension of a parametric lattice $\Lambda_t$ may vary with $t$, but it will eventually stabilize. This is why in the definition of a parametric vector $\fvec(t)$ being ``in'' a parametric lattice, we only require that $\fvec(t) \in \Lambda_t$ for sufficiently large values of $t$: this allows us to smoothly handle the case in which there are finitely many values of $t$ for which the dimension of $\Lambda_t$ is non-maximal.


\medskip
We now generalize the concept of an LLL-reduced basis to the parametric context. 

\begin{definition}
\label{def:evenLLL}
Fix $1/4 < \delta < 1$ (usually $\delta = 3/4$), and suppose that $\Bset = \{\fvec_1(t), \dots, \fvec_n(t)\}$ is a parametric lattice basis in which each $\fvec_i \in \Z[t]^m$. Let $\{\fvec^*_{1}(t), \dots, \fvec^*_{n}(t)\}$ be the Gram-Schmidt reduced basis (over $\Q(t)$) and let $$\rho_{i,j}(t) = \frac{\langle \fvec_i(t), \fvec_j^*(t) \rangle}{\langle \fvec_j^*(t), \fvec_j^*(t) \rangle}.$$

Then $\Bset$ is \emph{eventually LLL-reduced} (with factor $\delta$) if each of the following two conditions hold for all sufficiently large values of $t$:
  \begin{itemize}
  \item (Eventually size reduced) For $1 \leq j < i \leq n$,  $ |\rho_{i,j} (t)| \leq 1/2$. 
  \item (Eventual Lov\'asz condition) For $2 \leq i \leq n$,  $ \frac{\| \fvec^*_i(t) \|^2}{\|\fvec^*_{i-1}(t) \|^2} + \rho_{i,i-1}(t)^2 \geq \delta.$
  \end{itemize}
\end{definition}

A slightly weaker condition on a parametric basis is being ``asymptotically LLL-reduced,'' which means that the corresponding conditions hold in the limit instead of eventually. More concretely:

\begin{definition}
\label{def:asymLLL1}
With $\Bset = \{\fvec_1(t), \dots, \fvec_n(t)\}$ as above, we say that $\mathcal{B}$ is \emph{asymptotically LLL-reduced} (with factor $\delta$, $1/4 < \delta < 1$) if the following two conditions hold:
  \begin{itemize}
  \item (Asymptotically size reduced) For $1 \leq j < i \leq n$, $\lim_{t \rightarrow \infty} |\rho_{i,j} (t)| \leq 1/2$. 
  \item (Asymptotic Lov\'asz condition) For $2 \leq i \leq n$,  $\lim_{t \rightarrow \infty} \left( \frac{\| \fvec^*_i(t) \|^2}{\|\fvec^*_{i-1}(t) \|^2} + \rho_{i,i-1}(t)^2 \right) \geq \delta.$
  \end{itemize}
 Note that we allow the limit on the left-hand side of the asymptotic Lov\'asz condition to be $\infty$.

\end{definition}

For a general EQP lattice, we say that it is eventually (or asymptotically) LLL-reduced just in case it is LLL-reduced along each ``branch.'' More precisely:

\begin{definition}
\label{def:asymLLL}
(Eventually or asymptotically LLL-reduced, general case)  Fix $1/4 < \delta < 1$, and suppose that $\Bset = \{\fvec_1(t), \dots, \fvec_n(t)\}$ is a parametric lattice basis in which each $\fvec_i $ is EQP and let $M \in \N$ be such that every $\fvec_i$ is eventually $M$-periodic. For each $i \in \{1, \ldots, n\}$ and $r \in \{0, \ldots, M-1\}$, let $\fvec_{i,r}(t) = \fvec_i(Mt + r)$, so that $\fvec_{i,r}(t) \in \Z[t]^m$ (at least for sufficiently large values of $t$).

Then $\Bset$ is \emph{eventually LLL-reduced} (with factor $\delta$) if for every $r \in \{0, \ldots, M-1\}$, the basis $\{\fvec_{1,r}(t), \ldots, \fvec_{n,r}(t)\}$ is eventually LLL-reduced according to Definition~\ref{def:evenLLL}, and similarly for \emph{asymptotically LLL-reduced}.
\end{definition}

In our arguments below, we will never have to refer to the more complicated definition above since we will work along each branch separately so that Definition~\ref{def:evenLLL} or \ref{def:asymLLL1} applies.

\subsection{Parametric Gram-Schmidt reduction}
We begin our development by recording two lemmas on the parametric version of Gram-Schmidt reduction which will be useful later.

\begin{lem} \label{lem:GSoutput} Let $\bvec_1, \dots, \bvec_n \in \Q[t]^m$ and let $q \in \Q$. Assume the vectors $\bvec_1(q), \dots, \bvec_n(q) \in \Q^m$ are linearly independent. Then the following coincide:
  \begin{enumerate} \item The output of the Gram-Schmidt algorithm applied to the vectors $\bvec_1(q), \dots, \bvec_n(q) \in \Q$, and
  \item The output of the Gram-Schmidt algorithm applied to the vectors $\bvec_1(t), \dots, \bvec_n(t)$ in $\Q(t)$ evaluated at $t=q$.
 \end{enumerate}
\end{lem}

\begin{proof}
  First, observe that $\bvec_1, \dots, \bvec_n$ are necessarily linearly independent in $\Q[t]$. If not, any linear dependence among them would also hold for the evaluated vectors $\bvec_1(q), \dots, \bvec_n(q)$, contradicting the hypothesis. So neither application of the Gram-Schmidt algorithm will require division by zero. 

  For each arithmetic operation we have a commutative diagram such as
  \[  \begin{tikzcd}
      \Q(t) \times \Q(t) \arrow[r, "+"] \arrow[d, "\ev_q \times \ev_q"] & \Q(t) \arrow[d, "ev_q"] \\ \Q \times \Q \arrow[r, "+"] & \Q    
    \end{tikzcd} \]
for addition. By composing an appropriate sequence of such diagrams, we see that the entire Gram-Schmidt algorithm commutes with evaluation.   
\end{proof}

\begin{prop} \label{prop:rho-bounds}
  Let $\fvec_1, \ldots, \fvec_n$ be a list of parametric vectors. Let $\fvec_1^\ast, \ldots, \fvec_n^\ast \in \Q(t)$ be the associated Gram-Schmidt basis and $\rho_{k,j}$ be the Gram-Schmidt coefficients for $1 \leq j < k \leq n$. Fix $k \in \{1, \ldots, n\}$ and apply one round of size reduction as in Algorithm \ref{alg:LLL}, letting $\fvec'_k = \fvec_k - \nearest{\rho_{k,j}} \fvec_j$ and let $\rho'_{k,i}$ be the updated Gram-Schmidt coefficient for each $i < k$. Then
  \begin{enumerate}
  \item $\fvec_1^\ast, \ldots, \fvec_n^\ast$ continue to be the Gram-Schmidt vectors associated to the new basis $\fvec_1, \ldots, \fvec'_k, \ldots, \fvec_n$,
  \item for \textbf{all} values of $t \in \N$, $| \rho'_{k,j}(t) | \leq \textup{min} \left( 1/2, | \rho_{k,j}(t) | \right)$,
  \item $\rho'_{k,i} =  \rho_{k,i}$ whenever $j < i < k$, and
  \item $\rho'_{\ell, i} = \rho_{\ell, i}$ whenever $\ell \neq k$.
  \end{enumerate}
\end{prop}

\begin{proof}
Essentially the same statements are given for vectors in $\Z^m$ in \cite[Lemma 17.4.1 and Exercise 17.4.8]{Galbraith}, and the fact that they extend to parametric vectors follows from Lemma \ref{lem:GSoutput}. The only difference from the lemmas in \cite{Galbraith} is that in (2) we add the additional requirement that $|\rho'_{k,j}(t)| \leq |\rho_{k,j}(t)|$, but this follows from considering that if $| \rho_{k,j} | < 1/2$ then clearly $\fvec'_k = \fvec_k$ and hence $| \rho'_{k,j}(t)| = |\rho_{k,j}(t)|$. Condition (4) follows from the definitions plus part (1), since $\rho_{\ell, i}$ depends only upon the values of $\fvec^*_i$ and $\fvec_\ell$, neither of which is changed.

\end{proof}

\subsection{Pilot vectors}
One of the keys to generalizing the LLL algorithm to the case of parametric vectors is to identify an important special case in which the original algorithm directly applies. The idea is that if the parametric vectors $f_1, \dots, f_n$ all have the same degree, then up to some technical details, we can get away with applying the LLL algorithm not to $f_1, \dots, f_n$ themselves, but to their leading coefficient or \emph{pilot} vectors, as defined below, provided that these pilot vectors are linearly independent. 

\begin{definition} Let $\fvec = (f^{(1)}, \dots, f^{(m)}) \in \Q(t)^m$ be a parametric vector.
  \begin{enumerate}
  \item The \emph{degree} of $\fvec$ is $\deg(\fvec) := \max\{\deg(f^{(1)}, \dots, \deg(f^{(n)})) \}$.
  \item For each $0 \leq d \leq \deg(\fvec)$, the \emph{degree $d$ part} of $\fvec$ is $\fvec^d = (f_d^{(1)}, \dots f_d^{(m)}) \in \Z^m$, where $f_d^{(k)}$ is the coefficient of $t^d$ in $f^{(k)}$ for $1 \leq k \leq m$.
  \item The \emph{pilot vector} of $\fvec$ is $\widetilde{\fvec} = \fvec^{\deg{\fvec}}$.
  \item If $S$ is a set of parametric vectors over $\Z[t]^m$, then $\widetilde{S} \subseteq \Z^m$ is the set of pilot vectors of parametric vectors in $S$.
  \end{enumerate}
 We may similarly define the pilot vector in the case where $\fvec$ is mild EQP (Definition \ref{def:mildEQP}), as the degree and leading coefficients of $f^{(i)}$ are (eventually) well-defined.
\end{definition}

\begin{remark} \label{rem:guiding}
  As $t$ tends to infinity, the direction of $\fvec$ approaches the direction of $\widetilde{\fvec}$ and in fact $\lim_{t \to \infty} \frac{\|\fvec(t) - \widetilde{\fvec}(t) \|}{\|\fvec(t)\|} = 0$ so the pilot vector $\widetilde{\fvec}$ indeed ``guides'' the parametric vector $\fvec$. 
 
  Note also that entries of $\fvec$ which are not of maximal degree will correspond to zero entries of $\widetilde{\fvec}$, but $\widetilde{\fvec}$ will only be the zero vector if $\fvec$ is.
\end{remark}
  Pilot vectors behave well with respect to the Gram-Schmidt algorithm in the following sense.

\begin{lem} \label{lem:GSpilot}
  Let $\fvec_1, \dots, \fvec_n$ be parametric vectors over $\Q(t)$ such that that $\deg(\fvec_1) \leq \ldots \leq \deg(\fvec_n)$, and suppose that their respective pilot vectors $\widetilde{\fvec_1},\dots, \widetilde{\fvec_n} \in \Q^m$ are linearly independent over $\Q$. Let $\fvec_1^\ast, \dots, \fvec_n^\ast$ be the Gram-Schmidt vectors obtained from $\fvec_1,\dots,\fvec_n$ over $\Q(t)$ and $\rho_{i,j} \in \Q(t)$ be the Gram-Schmidt coefficients. Let $(\widetilde{\fvec_1})^\ast, \dots, (\widetilde{\fvec_n})^\ast$ be the Gram-Schmidt vectors over $\Q$ obtained from $\widetilde{\fvec_1}, \dots, \widetilde{\fvec_n}$ and $\{\mu_{i,j}\}$ be the Gram-Schmidt coefficients. Then 
  \begin{enumerate}
  \item $\deg(\fvec_i^\ast) = \deg(\fvec_i)$ for $1 \leq i \leq n$;
  \item Let $d_i = \deg(\fvec_i)$. Then $$\lim_{t \to \infty} \frac{\rho_{i,j}}{t^{d_i - d_j}} = \mu_{i,j}$$ for $1 \leq j < i \leq n$; 
  \item Either for all sufficiently large $t$, $$  \frac{\rho_{i,j}}{t^{d_i - d_j}} \geq \mu_{i,j},$$ or else for all sufficiently large $t$, $$ \frac{\rho_{i,j}}{t^{d_i - d_j}} \leq \mu_{i,j} ;$$ and 
  \item $\widetilde{(\fvec_i^\ast)} = (\widetilde{\fvec_i})^\ast$ for $1 \leq i \leq n$.
  \end{enumerate}
\end{lem}

\begin{proof}
 We prove all four statements by simultaneous strong induction on $i$. For $i=1$, (2) and (3) vacuously true and (1) and (4) are immediate because $\fvec_1^\ast = \fvec_1$ and $(\widetilde{\fvec_1})^\ast = \widetilde{\fvec_1}$. So suppose all three statements are true up to $i-1$ and (in the case of (2)) for all appropriate values of $j$. We recall that 
  \begin{equation}\tag{*} \fvec_i^\ast = \fvec_i - \sum_{j<i}\rho_{i,j}\fvec_j^\ast. \label{eq:GSpilot} \end{equation} 

  To prove (1), we know by induction that $\deg(\fvec_j^\ast) = \deg(\fvec_j) = d_j$ for each $j < i$, so 
  
  $$\rho_{i,j}(t) = \frac{ \langle \fvec_i(t), \fvec_j^\ast(t) \rangle}{\langle \fvec_j^\ast(t), \fvec_j^\ast(t) \rangle} = \frac{\langle \widetilde{\fvec_i} t^{d_i} + \cvec(t), \widetilde{\fvec_j^\ast} t^{d_j} + \dvec(t) \rangle}{\langle \widetilde{\fvec_j^\ast} t^{d_j} + \dvec(t) , \widetilde{\fvec_j^\ast} t^{d_j} + \dvec(t) \rangle}$$
  where $\cvec(t) \in \Q(t)^m$ is of degree less than $\dvec_i(t) \in \Q(t)^m$ is of degree less than $d_j$. It follows that $\rho_{i,j}$ is rational of degree at most $d_i - d_j$. Write $\rho_{i,j} = c_{i,j} t^{d_i - d_j} + \sigma_{i,j}$ where $c_{i,j} \in \Q$ and $\sigma_{i,j}$ is a rational function of degree strictly less that $d_i - d_j$. Then the degree-$d_i$ part of $\rho_{i,j}\fvec_j^\ast$ is just $c_{i,j}\widetilde{(\fvec_j^\ast)}$ (since everything else has strictly lower degree and $\deg(\fvec_j^\ast) = d_j$.) But by (4) applied to $j$, $\widetilde{(\fvec_j^\ast)} = (\widetilde{\fvec_j})^\ast$ which is a linear combination of $\{\widetilde{\fvec_1}, \dots, \widetilde{\fvec_j}\}$. It follows that the degree-$d_i$ part of the entire right-hand side of (\ref{eq:GSpilot}) is $\widetilde{\fvec_i}$ plus a linear combination of $\{\widetilde{\fvec_1}, \dots, \widetilde{\fvec_{i-1}}\}$, which cannot be zero by the hypothesis of linear independence. So $\deg(\fvec_i^\ast) = d_i$.

  For (2), we directly calculate
  \[ \begin{array}{lllllll} \lim_{t \to \infty} \frac{\rho_{i,j}(t)}{t^{d_i - d_j}}
    & = & \lim_{t \to \infty} \frac{\langle \fvec_i, \fvec_j^\ast \rangle }{\langle \fvec_j^\ast, \fvec_j^\ast \rangle} t^{d_j - d_i} 
    & = & \lim_{t \to \infty} \frac{\langle \fvec_i / t^{d_j}, \fvec_j^\ast / t^{d_j} \rangle }{ \langle \fvec_j^\ast / t^{d_j}, \fvec_j^\ast / t^{d_j} \rangle} t^{d_j - d_i} 
    & = & \lim_{t \to \infty} \frac{\langle \fvec_i / t^{d_i}, \fvec_j^\ast / t^{d_j} \rangle }{ \langle \fvec_j^\ast / t^{d_j}, \fvec_j^\ast / t^{d_j} \rangle} \\
    & = & \frac{\langle \widetilde{\fvec_i}, \widetilde{\fvec_j^\ast} \rangle}{\langle \widetilde{\fvec_j^\ast}, \widetilde{\fvec_j^\ast} \rangle} 
    & = & \frac{\langle \widetilde{\fvec_i}, (\widetilde{\fvec_j})^\ast \rangle}{\langle (\widetilde{\fvec_j})^\ast, (\widetilde{\fvec_j})^\ast \rangle} 
    & = &\mu_{i,j},  
  \end{array} \]
  where the fourth equality follows from (1) and the fifth equality from (4) applied to $j$.

Clause (3) is now an immediate consequence of (2), since $\frac{\rho_{i,j}}{t^{d_i - d_j}}$ is a rational function and so it is either eventually greater than or equal to, or else eventually less than or equal to, any given constant (and $\mu_{i,j}$ in particular).

  For (4), we take the pilot vectors on both sides of equation (\ref{eq:GSpilot}). By (2), each side has degree $d_i$. As in the proof of (1), the degree-$d_i$ part of $\rho_{i,j}\fvec_j^\ast$ is $c_{i,j} (\widetilde{\fvec_j})^\ast$. By (2), we also have $c_{i,j} = \mu_{i,j}$. Thus we obtain
  \[ \widetilde{\fvec_i^\ast} = \widetilde{\fvec_i} - \sum_{j < i} \mu_{i,j} (\widetilde{\fvec_j})^\ast.\]
  But the right-hand side of this equation is exactly $(\widetilde{\fvec_i})^\ast$ by definition.
\end{proof}

We can now show that for a collection $\{\fvec_1, \dots, \fvec_n\}$ of parametric vectors of the same degree with linearly independent pilot vectors, we can obtain an asymptotically LLL-reduced basis for their span by simply applying the LLL algorithm to their pilot vectors. More precisely:  

\begin{prop}
\label{prop:samedegree}
  Let $\fvec_1, \dots, \fvec_n \in \Z[t]$ be all of the same degree $d$, and suppose that their pilot vectors $\widetilde{\fvec_1},\dots,\widetilde{\fvec_n} \in \Z$ are linearly independent over $\Q$. Let $\bvec_1,\dots,\bvec_n$ be an LLL-reduced basis (with factor $\delta$) of the lattice $L_d = \Span_\Z\{\widetilde{\fvec_1},\dots,\widetilde{\fvec_n}\}$, and let $U_0 \in \Z^{n \times n}$ be a unimodular matrix such that
  \[ [\widetilde{\fvec_1},\dots,\widetilde{\fvec_n} ] = [\bvec_1 \dots \bvec_n] U_0\]
  as in Corollary \ref{cor:constructiveLLL}.
  Then the parametric vectors $\gvec_1,\dots,\gvec_n$ obtained by 
  \[ [\gvec_1 \dots \gvec_n] =  [\fvec_1 \dots \fvec_n] U_0\] 
  form an asympotically LLL-reduced basis (with the same factor $\delta$) for the lattice $\Lambda_d = \Span_\Z\{\fvec_1,\dots,\fvec_n\}$.
\end{prop}

\begin{proof}
  Since $\widetilde{\fvec_1},\dots,\widetilde{\fvec_n} \in \Z$ are linearly independent and $U_0$ is invertible, $\gvec_1,\dots,\gvec_n$ are all of degree $d$ and their respective pilot vectors are $\bvec_1,\dots,\bvec_n$.
  
We now verify the size-reduced condition. Let $\{\rho_{i,j}\}$ be the Gram-Schmidt coefficients for the vectors $\gvec_1, \dots, \gvec_n$ and $\{\mu_{i,j}\}$ be the Gram-Schmidt coefficients for $\widetilde{\gvec_1}, \dots, \widetilde{\gvec_n}$. Then for each $i$ and each $j < i$, by Lemma \ref{lem:GSpilot} we have $\lim_{t \to \infty} \rho_{i,j} = \mu_{i,j}$. But since $\widetilde{\gvec_i} = \bvec_i$ and $\bvec_1,\dots,\bvec_n$ form an LLL-reduced basis, $\mu_{i,j} \leq 1/2$. So $\lim_{t \to \infty} |\rho_{i,j} (t)| \leq 1/2$.

Finally we verify the Lov\'asz condition. Using Remark \ref{rem:guiding} and the various parts of Lemma \ref{lem:GSpilot}, we see that for each $2 \leq i \leq n$, 
\begin{align*}
  \lim_{t \rightarrow \infty} \frac{\| \gvec^*_i(t) \|^2}{\|\gvec^*_{i-1}(t) \|^2} + \rho_{i,i-1}(t)^2 & = \left( \lim_{t \rightarrow \infty} \frac{\| \gvec^*_i(t) \|^2}{\|\gvec^*_{i-1}(t) \|^2}  \right) + \mu^2_{i,i-1} \\
  & =  \frac{\| \widetilde{\gvec^*_i}(t) \|^2}{\|\widetilde{\gvec^*_{i-1}}(t) \|^2}  + \mu^2_{i,i-1} \\
  & =  \frac{\| (\widetilde{\gvec_i})^*(t) \|^2}{\| (\widetilde{\gvec_i})^*(t) \|^2}   + \mu^2_{i,i-1} \\
  & =   \frac{\| \bvec_i^\ast(t) \|^2}{\| \bvec_{i-1}^\ast(t) \|^2}  + \mu^2_{i,i-1},
\end{align*}
and this last expression is greater than or equal to $\delta$ because $\{ \bvec_1,\dots,\bvec_n \}$ is LLL-reduced.
\end{proof}

The previous Proposition yields an \textbf{asymptotically} LLL-reduced basis, but we would like an \textbf{eventually} LLL-reduced basis. It is not quite true that any unimodular matrix $U_0$ which gives an LLL-reduced basis for the pilot vectors will always give an eventually LLL-reduced basis when applied to the original vectors in the case where one of the quantities $|\mu_{i,j}|$ tends to exactly $1/2$.

\begin{example}Let $\fvec_1 = (2t, 2t)$ and $\fvec_2 = (0, t+1)$. Then $\fvec_1^* = \fvec_1$, $$\rho_{2,1} = \frac{2 t^2 + 2t}{4 t^2} = \frac{1}{2} + \frac{1}{2t},$$ and $\widetilde{\fvec_2^*} = (-1, 0)$. It is simple to check that $\{\widetilde{\fvec_1}, \widetilde{\fvec_2}\}$ is asymptotically LLL-reduced, so that in Proposition~\ref{prop:samedegree} we could take $U = I_2$, but $\rho_{2,1}$ is always slightly greater than $\frac{1}{2}$.
\end{example}

For this reason, to obtain eventually LLL-reduced bases in the case where the original basis has uniform degree, we will need to apply the following Lemma after applying Proposition~\ref{prop:samedegree}:

\begin{lem}
\label{lem:eventualsizereduction}
 Let $\gvec_1, \dots, \gvec_n \in \Z[t]$ be all of the same degree $d$, and suppose that they form an asymptotically LLL-reduced basis for the lattice they span. Furthermore, suppose that for some $j, k$ with $1 \leq j < k \leq n$, $$\lim_{t \rightarrow \infty} |\rho_{k,j} | = \frac{1}{2}$$ but for all sufficiently large values of $t$, $$|\rho_{k,j} | > \frac{1}{2}.$$ Then we can replace $\fvec_k$ by $\fvec'_k = \fvec_k \pm \fvec_j$ (where the choice of sign depends on the signs of the leading coefficients of $\fvec_k$ and $\fvec_j$) in such a way that:
 \begin{enumerate}

  \item $\rho'_{k,i} = \rho_{k,i}$ whenever $j < i < k$ (where ``$\rho_{k,j}'$'' is the quantity $\rho_{k,j}$ computed in terms of the new basis), and $\rho'_{\ell,i} = \rho_{\ell,i}$ whenever $\ell \neq k$; 
 \item for all sufficiently large $t$, $$|\rho'_{k,j} | < \frac{1}{2};$$ and
 \item  $\{\fvec_1, \ldots, \fvec'_k, \ldots, \fvec_n\}$ still satisfies the asymptotic Lov\'asz condition (with the same value of $\delta$).

 \end{enumerate}
\end{lem}

\begin{proof}
By our hypotheses, $1 \geq |\rho_{k,j}| > 1/2$ for all sufficiently large $t$, so $\nearest{\rho_{k,j}} \in \{-1, 1\}$, and we choose $$\fvec'_k = \fvec_k - \nearest{\rho_{k,j}} \fvec_j.$$

Let the superscript ${}'$ denote quantities computed with respect to the new basis $\{\fvec_1, \ldots, \fvec'_k, \ldots, \fvec_n\}$ (for instance, ($\fvec^*_i)'$ or $\rho'_{i,\ell}$). By Proposition~\ref{prop:rho-bounds} (1), $(\fvec^*_i)' = \fvec^*_i$ for any $i$. Part (1) of the Lemma is just parts (3) and (4) of Proposition~\ref{prop:rho-bounds}. 

Proof of (2): Note that $$| \rho'_{k,j} | = \left| {\frac{\langle \fvec_k - \nearest{\rho_{k,j}} \fvec_j, \fvec^*_j \rangle}{\langle \fvec^*_j, \fvec^*_j \rangle}} \right| = \left| {\frac{\langle \fvec_k , \fvec^*_j \rangle}{\langle \fvec^*_j, \fvec^*_j \rangle}} - \nearest{\rho_{k,j}} \frac{ \langle \fvec_j, \fvec^*_j \rangle }{\langle \fvec^*_j, \fvec^*_j \rangle}  \right| = \left| \rho_{k,j} - \nearest{\rho_{k,j}} \right|$$ (where the last equality uses the identity $\langle \fvec_j, \fvec^*_j \rangle = \langle \fvec^*_j, \fvec^*_j \rangle$, see \cite{Galbraith} Exercise 17.4.8 (1)). From this and our hypotheses on $\rho_{k,j}$, it immediately follows that $|\rho'_{k,j}| < 1/2$ for all sufficiently large $t$, and also that $\lim_{t \rightarrow \infty} | \rho'_{k,j}| = 1/2$.

Proof of (3): To check the asymptotic Lov\'asz condition $$ \lim_{t \rightarrow \infty} \frac{\| \fvec^*_i(t) \|^2}{\|\fvec^*_{i-1}(t) \|^2} + \rho'_{i,i-1}(t)^2 \geq \delta$$ for the new basis, we consider three cases, depending on whether $i < k$, $i = k$, or $i > k$. If $i < k$, then $\rho'_{i,i-1} = \rho_{i,i-1}$ Proposition~\ref{prop:rho-bounds} (4), so trivially the Lov\'asz condition is preserved. If $i = k$, then we must consider two subcases: first, if $j < i-1$, then we can apply Proposition~\ref{prop:rho-bounds} (3) to conclude that $\rho'_{i, i-1} = \rho_{i,i-1}$, and so the Lov\'asz inequality is preserved; on the other hand, if $i = k$ and $j = i-1$, then as noted in the previous paragraph, $\lim_{t \rightarrow \infty} |\rho'_{i, i-1}| = 1/2 = \lim_{t \rightarrow \infty} |\rho_{i, i-1}|$, so again the Lov\'asz condition is preserved. Finally, if $i > k$, then $\rho'_{i, i-1} = \rho_{i, i-1}$ by part (1) of the Lemma, so once again the Lov\'asz inequality is preserved.
\end{proof}

\begin{prop}
\label{prop:samedegree2}
  Let $\fvec_1, \dots, \fvec_n \in \Z[t]$ be all of the same degree $d$, fix $\delta$ ($1/4 < \delta < 1$), and suppose that their pilot vectors $\widetilde{\fvec_1},\dots,\widetilde{\fvec_n} \in \Z$ are linearly independent over $\Q$. Then there is a unimodular matrix $U \in \Z^{n \times n}$ such that the parametric vectors $\gvec_1,\dots,\gvec_n$ obtained by 
  \[ [\gvec_1 \dots \gvec_n] =  [\fvec_1 \dots \fvec_n] U\] 
  form a basis for the lattice $\Lambda_d = \Span_\Z\{\fvec_1,\dots,\fvec_n\}$ which satisfies the \textbf{eventual} size-reduced condition and the \textbf{asymptotic} Lov\'asz condition with factor $\delta$.
\end{prop}

\begin{proof}
Begin with the unimodular matrix $U_0$ from Proposition~\ref{prop:samedegree} above to obtain a new basis $\{\hvec_1, \ldots, \hvec_n\}$ which is asymptotically LLL-reduced. If $\{\hvec_1, \ldots, \hvec_n\}$ is not eventually size-reduced, then by Lemma~\ref{lem:eventualsizereduction} we may perform further size reductions until it is eventually size-reduced: the reductions are applied to $\hvec_2, \ldots, \hvec_k, \ldots, \hvec_n$ in turn, and when reducing a vector $\hvec_k$, we consider reductions by $\hvec_{k-1}, \ldots, \hvec_{ j}, \ldots, \hvec_{ 1}$ in decreasing order by $j$. By clause (1) of Lemma~\ref{lem:eventualsizereduction}, this order of reductions will maintain all of the conditions on the basis that we care about. Since we are only adding or subtracting single basis vectors, the final output $\gvec_1,\dots,\gvec_n$ is the result of multiplying the input by some unimodular matrix $U$.
\end{proof}

In our eventual parametric LLL reduction algorithm, we will order the vectors by degree. This has the advantage that the eventual L\'ovasz condition is immediate whenever we compare vectors of different degree, as follows.

\begin{lem}
\label{lem:lovasz}
If $\{ \fvec_1, \ldots, \fvec_n\}$ is a parametric lattice basis over $\Z[t]$, $\deg(\fvec_i) = e$, and:
\begin{enumerate}
\item For every $j < i$, $\deg(\fvec_j) < e$; and
\item $\widetilde{\fvec_i} \notin \Span (\widetilde{\fvec_1}, \ldots, \widetilde{\fvec_{i-1}} )$,
\end{enumerate}
then the eventual Lov\'asz condition holds between $\fvec_i$ and $\fvec_{i-1}$ (for \emph{any} value of $\delta$).
\end{lem}

\begin{proof}
  By (1) and Lemma \ref{lem:GSpilot}, $\deg(\fvec_{i-1}^\ast) < e$ and by (2) and Lemma \ref{lem:GSpilot}, $\deg(\fvec_i^\ast) = e$. Thus
  \[ \frac{\| \fvec^*_i(t) \|^2}{\|\fvec^*_{i-1}(t) \|^2} + \rho_{i,i-1}(t)^2 \geq \frac{\| \fvec^*_i(t) \|^2}{\|\fvec^*_{i-1}(t) \|^2} \]
  which tends to infinity. 
 \end{proof}

Noting that Lemma \ref{lem:GSpilot} and Proposition \ref{prop:samedegree} require linearly independent pilot vectors in their hypotheses, we now show that we can modify a parametric lattice basis in order to achieve linear independence. 

\begin{lem} \label{lem:Hermite}
  Let $\fvec_1, \dots, \fvec_n \in \Z[t]^m$ be all of the same degree $d$. Then there exist $\gvec_1, \dots, \gvec_n \in \Z[t]^m$ and some $1 \leq j \leq n$ such that:
  \begin{enumerate}
  \item $\Span_\Z\{\gvec_1,\dots,\gvec_n\} = \Span_\Z\{\fvec_1,\dots,\fvec_n\}$, 
  \item $\deg(\gvec_1) = \dots = \deg(\gvec_j) = d$,  
  \item $\deg(\gvec_{j+1}) = \dots = \deg(\gvec_n) < d$, and 
  \item the pilot vectors $\{\widetilde{\gvec_1}, \dots, \widetilde{\gvec_j}\}$ are linearly independent over $\Q$.   
  \end{enumerate}
\end{lem}

  \begin{proof}
    Let $\avec_i \in \Z^m$ be the pilot vector of $\fvec_i$ for $i=1,\dots,n$ and $A \in \Z^{m \times n}$ be the matrix whose columns are $\avec_1, \dots, \avec_n$. Let $B$ be the column-style Hermite normal form of $A$ and $\bvec_1, \dots, \bvec_n$ be the columns of $B$. Since Hermite normal form is obtained by an invertible transformation over $\Z$, there exists a unimodular matrix $U \in \Z^{n \times n}$ such that $AU = B$. By definition of Hermite normal form, for some $1 \leq j \leq n$ the first $j$ columns of $B$ are linearly independent over $\Q$ and the remaining columns consist entirely of zeros.
    
  We now apply the same transformation to the original vectors of polynomials. That is, let $F \in \Z^{m \times n}$ be the matrix whose columns are $\fvec_1, \dots, \fvec_n$ and $\gvec_1, \dots, \gvec_n$ be the columns of the matrix $G = FU$. Then (1) holds because $U$ is unimodular, (2) and (4) because $\bvec_1, \dots, \bvec_j$ are linearly independent over $\Q$, and (3) because $\bvec_{j+1} = \dots = \bvec_n = \mathbf{0}$. 
\end{proof}

\begin{prop} \label{prop:Hermite}
  Let $\fvec_1, \ldots,  \fvec_n \in \Z[t]^m$, ordered by degree ($\deg(\fvec_1) \leq \deg(\fvec_2) \leq \ldots$). Then there exist $\gvec_1, \ldots, \gvec_n \in \Z[t]^m$, also ordered by degree, such that:
  \begin{enumerate}
  \item $\Span_\Z\{\gvec_1,\dots,\gvec_n\} = \Span_\Z\{\fvec_1,\dots,\fvec_n\}$;
   \item for each $d \in \{0, 1, \ldots \} \cup \{-\infty\}$, if $\mathcal{B}_d$ is the set of all $\gvec_i$ of degree $d$, then $\widetilde{\mathcal{B}_d}$ is linearly independent over $\Q$; 
   \item for every $i \in \{1, \ldots, n\}$, $\deg(\gvec_i) \leq \deg(\fvec_i)$; and
 \item $\deg(\fvec_i) = \deg(\gvec_i)$ for all $i$ just in case for every possible degree $d$, the set  $\{ \widetilde{\fvec_j} \, : \, \deg(\fvec_j) = d \}$ is linearly independent.

\end{enumerate}

    
\end{prop}

Note that in Proposition \ref{prop:Hermite} we do \emph{not} need to assume that the original parametric vectors $\fvec_1,\ldots, \fvec_n $ are $\Z$-linearly independent, and to deal with this case we adopt the convention that $\deg( 0) = -\infty$. Condition (4) will be involved in showing termination of our algorithm: we will only invoke Proposition \ref{prop:Hermite} when the pilot vectors in some degree are linearly dependent, thus reducing the sum of the degrees.

\begin{proof}
 To prove Proposition~\ref{prop:Hermite}, we apply Hermite normal form in each degree as in Lemma \ref{lem:Hermite}, from highest to lowest. By unimodularity, we have (1). Since degrees are never increased, we have (3). While working in a given degree $d$ we do not affect the vectors of degrees higher than $d$ and we create linearly independent pilot vectors in degree $d$; thus (2) follows recursively.

  Finally we must show each implication in (4). If for some degree $d$ the vectors of degree $d$ have linearly independent pilot vectors, then the unimodular matrix $U$ that converts these pilot vectors to Hermite normal form has no columns of zeros, so no vectors of lower degree are created. The right-to-left direction follows recursively.

  On the other hand, suppose that for some degree $d$ the pilot vectors are linearly dependent. Choose the largest such $d$, so that the reduction in degrees higher than $d$ does not create new vectors of degree $d$. Then in degree $d$, the matrix $U$ does have a column of zeros, which creates a vector of lower degree. 
\end{proof}


\subsection{Asymptotic orthogonality}

For our eventual algorithm, we need to arrange for the set of \emph{all} pilot vectors to be linearly independent, not just the pilot vectors arising from basis elements of the same degree (as Proposition~\ref{prop:Hermite} yields). To achieve this, we will show a stronger statement in Corollary \ref{cor:orthogonal_to_all}: given sets of vectors of two different degrees $d$ and $e$ we can perform a reduction that either lowers the degrees or produces a \emph{asymptotically orthogonal} sets, defined as follows. 

\begin{definition} \label{orth} 
\begin{enumerate}
\item Two parametric vectors $\fvec$ and $\gvec$ over $\Z[t]$ are \emph{asymptotically orthogonal} (written $\fvec \perp^{asym} \gvec$) if their pilot vectors $\widetilde{\fvec}$ and $\widetilde{\gvec}$ are orthogonal.
\item If $S$ and $T$ are sets of parametric vectors over $\Z[t]$, then we say that $S$ is \emph{asymptotically orthogonal} to $T$ (written $S \perp^{asym} T$) if every vector in $S$ is asymptotically orthogonal to every vector in $T$.
\end{enumerate}
\end{definition}

The following two lemmas will be important in our eventual parametric LLL reduction algorithm, and they are the main reason why we introduce the concept of asymptotic orthogonality.

\begin{lem}
\label{lem:asymptotic_orthogonality}
Suppose that $\mathcal{B} = \{\fvec_1, \ldots, \fvec_n\}$ is a parametric lattice basis which is sorted by degree (that is, if $i < j$ then $\deg(\fvec_i) \leq \deg(\fvec_j)$) and such that the set of pilot vectors $\widetilde{\mathcal{B}}$ is linearly independent. Let $\mathcal{B}_d$ be the subset of $\mathcal{B}$ consisting of all the degree-$d$ vectors, and assume that whenever $d \neq e$ we have $\mathcal{B}_d \perp^{asym} \mathcal{B}_e$.

Then if $i < j$ and $\deg(\fvec_i) < \deg(\fvec_j)$, we have $$\deg(\rho_{j,i}) < \deg(\fvec_j) - \deg(\fvec_i),$$ where the $\rho_{j,i}$ are the Gram-Schmidt coefficients computed from $\mathcal{B}$ over $\Q(t)$.
\end{lem}

\begin{proof}
Recall that $$\rho_{j,i} = \frac{\langle \fvec_j, \fvec_i^*\rangle}{\langle \fvec_i^*, \fvec_i^* \rangle}.$$ Let $d_i = \deg(\fvec_i)$ and $d_j = \deg(\fvec_j)$, and recall that by Lemma~\ref{lem:GSpilot}, $\deg(\fvec_i^*) = d_i$. We consider the numerator of $\rho_{j,i}$ first. Observe that $$\left( \langle \fvec_j, \fvec_i^* \rangle \right)^{d_i + d_j} = \langle \widetilde{\fvec_j}, \widetilde{\fvec_i^*} \rangle = \langle \widetilde{\fvec_j}, (\widetilde{\fvec_i})^* \rangle = \langle \widetilde{\fvec_j}, \sum_{k=1}^i c_k \widetilde{\fvec_k} \rangle$$ for some coefficients $c_k \in \Q$, where the second equality is by part (3) of Lemma~\ref{lem:GSpilot}. Since $\mathcal{B}$ is ordered by degree, by the assumption of asymptotic orthogonality between vectors of different degrees, we have that $\langle \widetilde{\fvec_j}, \widetilde{\fvec_k} \rangle = 0$ for any $k \in \{1, \ldots, i\}$. 

Thus $\left( \langle \fvec_j, \fvec_i^* \rangle \right)^{d_i + d_j} = 0$ and so the numerator of $\rho_{j,i}$ has degree strictly less than $d_j + d_i$. But the denominator has degree $2 d_i$ (since $\deg(\fvec_i^*) = \deg(\fvec_i) = d_i$), so $\deg(\rho_{j,i}) < d_j + d_i - 2 d_i = d_j - d_i$.
\end{proof}

The next lemma may be confusing without a clarification and an illustrative example. If $\mathcal{B} = \{\fvec_1, \ldots, \fvec_n\}$ is a parametric lattice basis and we take some sub-basis $\mathcal{B}' = \{\fvec_i, \fvec_{i+1}, \ldots, \fvec_j\}$ of consecutive vectors from $\mathcal{B}$, then the Gram-Schmidt vectors $\{\fvec^*_i, \ldots, \fvec^*_j\}$ as computed within the basis $\mathcal{B}'$ will generally be different from the corresponding Gram-Schmidt vectors $\{\fvec^*_1, \ldots, \fvec^*_n\}$ computed within $\mathcal{B}$. Hence there is a possibility that even if $\mathcal{B}'$ is asymptotically LLL-reduced, considered as a basis in its own right, two of the adjacent vectors from $\mathcal{B}'$ may no longer satisfy the asymptotic Lov\'asz condition. The example below shows that this is a real possibility:

\begin{example}
Consider the basis $\mathcal{B} = \{\fvec_1, \fvec_2, \fvec_3\}$ where $\fvec_1 = (t, 0, 0)^T$, $\fvec_2 = (0, 2t, 0)^T$, and $\fvec_3 = (t, t, t)^T$, and let $\mathcal{B}' = \{\fvec_2, \fvec_3\}$. First we check that $\mathcal{B}'$ is asymptotically Lov\'asz with $\delta = 3/4$. For this, we need to compute the Gram-Schmidt basis $\{\fvec_2^*, \fvec_3^*\}$ of $\mathcal{B}'$. Clearly $\fvec_2^* = \fvec_2$, while $$\fvec_3^* = \left(\begin{array}{c} t \\ t \\ t \end{array} \right) - \rho_{3,2} \left(\begin{array}{c} 0 \\ 2t \\ 0 \end{array} \right) = \left(\begin{array}{c} t \\ 0 \\ t \end{array}\right),$$ using the fact that $\rho_{3,2} = \frac{2 t^2}{4 t^2} = \frac{1}{2}$. Now to check the Lov\'asz condition, note that $$\frac{\| \fvec^*_3 \|^2}{\| \fvec_2^* \|^2} + \rho_{3,2}^2 = \frac{2 t^2}{4 t^2} + \left( \frac{1}{2} \right)^2 = \frac{3}{4},$$ as we wanted.

On the other hand, suppose we check the asymptotic Lov\'asz condition on the entire basis $\mathcal{B}$. In the corresponding Gram-Schmidt basis $\{\fvec_1^*, \fvec_2^*, \fvec_3^*\}$, clearly $\fvec_1^* = \fvec_1$ and $\fvec_2^* = \fvec_2$ (by orthogonality), while $$\fvec_3^* = \left(\begin{array}{c} t \\ t \\ t \end{array} \right) - \frac{2 t^2}{4 t^2} \left( \begin{array}{c} 0 \\ 2t \\ 0 \end{array} \right) - \left(\begin{array}{c} t \\ 0 \\ 0 \end{array} \right),$$ so $\fvec_3^* = (0, 0, t)^T$, with $\rho_{3,2} = \frac{1}{2}$ as before. But now the Lov\'asz condition for $\mathcal{B}$ between $\fvec_3$ and $\fvec_2$ would say that $$\lim_{t \rightarrow \infty}  \frac{\| \fvec^*_3 \|^2}{\| \fvec_2^* \|^2} + \rho_{3,2}^2  \geq \frac{3}{4},$$ which is false since $\| \fvec_3^* \|^2 = t^2$, $\|\fvec_2^* \|^2 = 4 t^2$, and $\rho_{3,2} = \frac{1}{2}$, so that the left-hand side is equal to the constant value of $\frac{1}{2}$.
\end{example}

The point of the next lemma is that as long as the parts of the basis with different degrees are asymptotically orthogonal to one another, then the asymptotic Lov\'asz condition will hold as long as it holds within the vectors of each degree.

\begin{lem}
\label{lem:Lovaszbydegree}
Suppose that $\mathcal{B}$ is a parametric lattice basis which is sorted by degree, the set of pilot vectors $\widetilde{\mathcal{B}}$ is linearly independent, and whenever $d \neq e$ we have $\mathcal{B}_d \perp^{asym} \mathcal{B}_e$ (where $\mathcal{B}_d, \mathcal{B}_e$ are the subsets consisting of degree-$d$ and degree-$e$ vectors, as above).
Then if the \textbf{asymptotic} Lov\'asz condition holds with factor $\delta'$ for each $\mathcal{B}_d$ separately, and if $1/4 < \delta < \delta' < 1$, then the \textbf{eventual} Lov\'asz condition holds for all of $\mathcal{B}$ with factor $\delta$.
\end{lem}

\begin{proof}
Write $\mathcal{B} = \{\fvec_1, \ldots, \fvec_n\}$ and let $d_i = \deg(\fvec_i)$. Fix $i$ with $1 < i \leq n$. Throughout, $\mathcal{B}^* = \{\fvec^*_1, \ldots, \fvec^*_n\}$ is the Gram-Schmidt basis corresponding to $\mathcal{B}$, so our goal is to verify that for all sufficiently large $t$,

\begin{equation}\label{Lov} \frac{\|\fvec_i^*(t) \|^2}{\| \fvec^*_{i-1}(t)\|^2} + \left(\frac{\langle \fvec_i(t), \fvec^*_{i-1}(t) \rangle}{\langle \fvec_{i-1}^*(t), \fvec_{i-1}^*(t) \rangle}\right)^2 \geq \delta.\end{equation}

First note that if $d_{i-1} < d_i$, then since the asymptotic orthogonality condition implies that $\widetilde{\fvec_i} \notin \Span(\widetilde{\fvec_1}, \ldots, \widetilde{\fvec_{i-1}} )$, Lemma~\ref{lem:lovasz} implies that the eventual Lov\'asz condition holds between $\fvec_i$ and $\fvec_{i-1}$. Therefore we may assume that $d_{i-1} = d_i$. If $\mathcal{B}_{d_i} = \{\fvec_k, \ldots, \fvec_\ell\}$, write $\mathcal{B}^+ = \{\fvec_k^+, \ldots, \fvec_\ell^+ \}$ for the Gram-Schmidt reduction of the basis $\mathcal{B}_{d_i}$, using the same indices for corresponding vectors in $\mathcal{B}$.

\begin{claim}
$(\widetilde{\fvec_k^+}, \ldots, \widetilde{\fvec_\ell^+} ) = (\widetilde{\fvec^*_k}, \ldots, \widetilde{\fvec_\ell^*}) $.
\end{claim}

\begin{proof}
It is straightforward to prove that $\widetilde{\fvec_j^+} = \widetilde{\fvec_j^*}$ by induction on $j \in \{k, \ldots, \ell\}$ using Lemma~\ref{lem:asymptotic_orthogonality}.
\end{proof}

Now since the degrees of $\fvec_i, \fvec_i^*,$ and $\fvec^*_{i-1}$ are all $d_i$ (by our assumptions and Lemma~\ref{lem:GSpilot}), the Claim above implies that to show that the inequality (\ref{Lov}) holds eventually, it is sufficient to check that

\begin{equation}\lim_{t \rightarrow \infty} \frac{\|\fvec_i^+(t) \|^2}{\| \fvec^+_{i-1}(t)\|^2} + \left(\frac{\langle \fvec_i(t), \fvec^+_{i-1}(t) \rangle}{\langle \fvec_{i-1}^+(t), \fvec_{i-1}^+(t) \rangle}\right)^2 \geq \delta'.\end{equation}

But this inequality is true by our hypothesis that $\mathcal{B}_{d_i}$ satisfies the asymptotic Lov\'asz condition with factor $\delta'$, so we are done.
\end{proof}

\begin{lem} \label{lem:orthogonal-span}
If $S$ and $T$ are sets of parametric vectors such that:
\begin{enumerate}
\item $S$ and $T$ are each homogeneous in degree: any two parametric vectors in $S$ have the same degree, and likewise for $T$;
\item Each set $\widetilde{S}$ and $\widetilde{T}$, considered separately, is linearly independent; and
\item $S \perp^{asym} T$,
\end{enumerate}
then $\Span_\Z(S) \perp^{asym} \Span_\Z(T)$.
\end{lem}

\begin{proof}
The fact that $\widetilde{S}$ is linearly independent and homogeneous in degree implies that for distinct $\fvec_1, \ldots, \fvec_n \in S$ and $a_1, \ldots, a_n \in \Z$, $$\widetilde{\sum_{i=1}^n a_i \fvec_i} = a_1 \widetilde{\fvec_1} + \ldots + a_n \widetilde{\fvec_n},$$ and likewise for $T$, so the Lemma follows from the corresponding fact about non-parametric vectors over $\Z$ (orthogonality is preserved by linear combinations).
\end{proof}

\begin{prop}
\label{prop:reduce_one_vector}
Suppose that $\hvec$ is a degree $e$ parametric vector, $\Bset$ is a set of parametric vectors over $\Z[t]$ which are all of the same degree $d < e$, and $\widetilde{\Bset}$ is linearly independent. Then there is an EQP parametric vector $\hvec'$ such that:

\begin{enumerate}
\item $\Span ( \Bset \cup \{\hvec'\}) = \Span ( \Bset \cup \{\hvec\})$;
\item If $\widetilde{\hvec} \in \Span_\Q( \widetilde{\Bset})$, then $\deg(\hvec') < \deg(\hvec)$; and
\item If $\widetilde{\hvec} \notin \Span_\Q( \widetilde{\Bset})$, then $\deg(\hvec') = \deg(\hvec)$, $\hvec'$ is mild EQP (Definition \ref{def:mildEQP}), and $\hvec'$ is asymptotically orthogonal to $\Bset$. 
\end{enumerate}
\end{prop}

\begin{proof}
  Let $\Bset = \{\fvec_1,\dots,\fvec_k\}$. Applying Gram-Schmidt over $\Q(t)$, we can write $\fvec_i^\ast = \fvec_i - \sum_{j=1}^{i-1} \rho_{i,j} \fvec_j^\ast$. So by induction we can write $\fvec_i^\ast =  \sum_{j=1}^{i} \beta_{i,j} \fvec_j$ where $\beta_{i,j} \in \Q(t)$. 

  We now extend the Gram-Schmidt process to $\Bset \cup \hvec$, obtaining $\hvec = \hvec^\ast + \gvec$, where $\gvec \in \Span_{\Q(t)}(\Bset)$ and $\hvec^\ast \in \Span_{\Q(t)}(\Bset)^\perp$ are both of degree $\leq e$. Explicitly, let $\sigma_i = \langle \hvec, \fvec_i^\ast \rangle / \langle \fvec_i^\ast, \fvec_i^\ast \rangle$, so that
  \begin{align*} \gvec & = \sum_{i=1}^k \sigma_i \fvec_i^\ast \\
    & = \sum_{i=1}^k \sigma_i \sum_{j=1}^i \beta_{i,j} \fvec_j \\
    & = \sum_{j=1}^k \left( \sum_{i=j}^k \beta_{i,j}\sigma_i \right) \fvec_j \\
    & = \sum_{j=1}^k \alpha_j \fvec_j
 \end{align*} 
where $\alpha_j = \sum_{i=j}^k \beta_{i,j}\sigma_i \in \Q(t)$. 

To obtain a decomposition of $\hvec$ over EQP which will be asymptotically close to the Gram-Schmidt decomposition, we just round off the coefficients $\alpha_j$. That is, let
\[ \hvec_1 := \sum_{j=1}^k \nearest{\alpha_j} \fvec_j , \]
\[ \hvec' = \hvec-\hvec_1.\]
Both are EQP by Lemma \ref{lem:closestinteger}, and $\hvec_1$ is mild EQP. This definition of $\hvec'$ immediately yields (1).  

\begin{claim} We have $\deg(\hvec^\ast) < e$ if and only $\widetilde{\hvec} \in \Span_{\Q} ( \widetilde{\Bset})$. 
\end{claim}

\begin{proof}
  Since $\deg(\hvec) = e$, we have
 \[ \widetilde{\hvec} = \hvec^e = (\hvec^\ast)^e + \gvec^e.\]
  If $\deg(\hvec^\ast)<e$, then $\widetilde{\hvec} = \gvec^e = \widetilde{\gvec} \in \Span_{\Q}(\widetilde{\Bset})$. On the other hand, if $\deg(\hvec^\ast) = e$, then $\widetilde{\hvec}$ is the sum of a nonzero element of $\widetilde{\Bset}^\perp$ and an element of $\Span_\Q(\widetilde{\Bset})$, so it cannot lie in $\Span_\Q(\widetilde{\Bset})$.  
\end{proof}

Now $(\hvec')^e = (\hvec-\hvec_1)^e = (\hvec-\gvec)^e+(\gvec-\hvec_1)^e$. But
$\gvec - \hvec_1 = \sum_{j=1}^k \left(\alpha_j - \nearest{\alpha_j} \right) \fvec_j$, the product of an eventually periodic function with a parametric vector of degree $d$, so $(\gvec-\hvec_1)^e = 0$. That is, $(\hvec')^e = (\hvec-\gvec)^e = (\hvec^\ast)^e$. If $\widetilde{\hvec} \in \Span_{\Q} ( \widetilde{\Bset})$, then $\deg(\hvec^\ast) < e$ by the Claim. But then $\deg(\hvec') < e$ as well, proving (2). Otherwise, $\deg(\hvec^\ast) = e$ by the Claim, and so $\deg(\hvec') = e$ as well. In particular, $(\hvec')^e = \hvec^e - \hvec_1^e$. Now $\hvec^e$ is constant by definition and $\hvec_1^e$ is constant by mildness, so $(\hvec')^e$ is also constant and hence $\hvec'$ is mild. Furthermore, $\widetilde{\hvec'} = \widetilde{\hvec^\ast}$ which is orthogonal to $\widetilde{\Bset}$ by Lemma \ref{lem:GSpilot}. Thus $\hvec'$ is asymptotically orthogonal to $\Bset$, proving (3).
\end{proof}

\begin{cor} \label{cor:orthogonal_to_all}
Say $e > d$. Suppose $\Bset$ is a set of parametric vectors over $\Z[t]$ and all of degree $d$ whose pilot vectors are linearly independent, $\Bset'$ is a set of parametric vectors over $\Z[t]$ and $\Bset, \Bset'$ are asymptotically orthogonal. Let $\hvec$ be a parametric vector of $\Z[t]$ of degree $e$ which is asymptotically orthogonal to $\Bset'$ and $\hvec'$ be the result of applying Proposition \ref{prop:reduce_one_vector} to $\hvec$ and $\Bset$. If $\deg(\hvec') = \deg(\hvec)$, then $\hvec'$ is asymptotically orthogonal to $\Bset'$.  
\end{cor}

\begin{proof}
  Recall that $\hvec' = \hvec - \hvec_1$ and $\hvec_1 = \sum_j \nearest{\alpha_j} \fvec_j$ where $\Bset = \{\fvec_1, \dots, \fvec_k\}$. We need to show that for each vector $\kvec_i \in \Bset'$, $\langle \widetilde{\hvec'}, \widetilde{\kvec_i} \rangle = 0$. If $\deg(\hvec_1) < e$, then $\widetilde{\hvec'} = \widetilde{\hvec}$ and we know this by hypothesis. Otherwise, $\widetilde{\hvec'} = \widetilde{\hvec} - \hvec_1^e$ and by the same hypothesis, it suffices to show that $\langle \hvec_1^e, \widetilde{\kvec_i} \rangle = 0$. But by the degree hypotheses, 
  \[ \displaystyle{\langle \hvec_1^e, \widetilde{\kvec_i} \rangle = \sum_j \nearest{\alpha_j}^{e-d} \langle \widetilde{\fvec_j}, \widetilde{\kvec_i} \rangle}, \]
  where $\alpha_j$ are as in the proof of Proposition \ref{prop:reduce_one_vector}. By the hypothesis that $\Bset, \Bset'$ are asymptotically orthogonal, each $ \langle \widetilde{\fvec_j}, \widetilde{\kvec_i} \rangle$ equals zero and we are done.  
\end{proof}

\subsection{A parametric LLL algorithm} \label{subsec:algorithm}
Now we describe our parametric version of the LLL reduction algorithm. The proof that this algorithm always terminates and gives an eventually LLL-reduced basis for a given parametric lattice will imply Theorem~\ref{thm:LLLreduced}.

The input to the parametric LLL reduction algorithm is a set $\mathcal{B} = \{\fvec_1, \ldots, \fvec_n\}$ of parametric vectors in $\Z[t]^m$ and a parameter $\delta$ (with $1/4 < \delta < 1$), and our final output will be an eventually LLL-reduced basis of the same parametric lattice (with factor $\delta$). Recall that we do \emph{not} assume the vectors in $\mathcal{B}$ to be linearly independent. Indeed, the algorithm below will weed out any extraneous $\fvec_i$ in Steps 1 and 2.

 Let $d$ be the maximum of the degrees of the $\fvec_i$.

\medskip

It will be useful in proving that the algorithm always terminates to define the \emph{degree sum} of a set $\{\fvec_1, \ldots, \fvec_n\}$ of nonzero parametric vectors as $$\textup{deg-sum}(\mathcal{B}) = \sum_i \deg( \fvec_i).$$

\underline{Step 1:} After applying Proposition~\ref{prop:Hermite} to $\mathcal{B}$, we may assume that:

\begin{enumerate}
\item If $\mathcal{B}_i \subseteq \mathcal{B}$ is the set of all vectors from $\mathcal{B}$ of degree $i$, then each set $\widetilde{\mathcal{B}_i}$ is linearly independent; and
\item $\mathcal{B}$ is ordered by degree: if $i \leq j$, then $\deg(\fvec_i) \leq \deg(\fvec_j)$.
\end{enumerate}

Note that if the input vectors $\{\fvec_1, \ldots, \fvec_n\}$ were linearly dependent over $\Z$, then some of the new vectors created in Step 1 will be $\textbf{0}$, that is, of degree $-\infty$. In this case, we immediately delete such vectors from our generating set in order to avoid any problem in defining its degree sum.

\medskip

\underline{Step 2:} We apply further reductions to the basis to conclude that we may assume, in addition to (1) and (2) above,

\begin{enumerate}
\setcounter{enumi}{2}
\item If $0 \leq d < e$, then $\mathcal{B}_d \perp^{asym} \mathcal{B}_e$.
\end{enumerate}

We will now show that condition (3) can be obtained by applying the reduction described in Proposition~\ref{prop:reduce_one_vector} repeatedly. Namely, suppose that $0 \leq k < d$ and the property (3) holds whenever $0 \leq i < j \leq k$, and say $\mathcal{B}_{k+1} = (\hvec_1, \ldots, \hvec_ \ell)$. For each $r$ staring from $0$ and working up to $k$, we apply Proposition~\ref{prop:reduce_one_vector} with $\mathcal{B} = \mathcal{B}_r$ on each vector $\hvec_s \in \mathcal{B}_{k+1}$, starting from $s=1$ and working up to $s=\ell$, and replace $\hvec_s$ by the new vector $\hvec'_s$ obtained from that Lemma. In case this new vector $\hvec'_s$ is $\textbf{0}$, we simply delete it from our generating set.

It may be that $\hvec'_s$ is EQP rather than polynomial, but if $\hvec'_s$ is eventually $M$-periodic we can branch on $t$ to assume that in fact $\hvec'_s \in \Z[t]^m$ (that is, perform a substitution of $Mt' + i$ for $t$ in all functions, for each $i \in \{0, \ldots, M-1\}$ in turn, and work with the new parameter variable $t'$). Note that after updating the value of $\hvec_s$ and branching if necessary, it may happen that either:

\medskip

 (A) $\deg(\hvec_s) \leq k)$, or
 
 \medskip
 
 (B) $\{\widetilde{\hvec_1}, \ldots, \widetilde{\hvec_s} \}$ becomes linearly dependent. 
 
 \medskip
 
 If either of these possibilities occurs, then we return to Step 1 above to resort all the vectors in $\mathcal{B}$ according to degree and ensure that pilot vectors within each degree are linearly independent.

To show that the process described in the previous paragraph does not give rise to an infinite loop (by returning to Step 1 infinitely many times), we must check that the situations described in (A) and (B) can only arise finitely many times. But each time (A) happens, $\textup{deg-sum}(\mathcal{B})$ must decrease; which can only happen a finite number of times; and likewise if (B) occurs, then the process of Hermite reduction as in Proposition~\ref{prop:Hermite} must also cause $\textup{deg-sum}(\mathcal{B})$ to go down (by condition (4) in Proposition~\ref{prop:Hermite}).

Finally, the result of applying the process above really does result in $\mathcal{B}_i \perp^{asym} \mathcal{B}_j$ for all $i < j$ by Corollary~\ref{cor:orthogonal_to_all}.

\medskip

Note as an immediate consequence that at this stage $\widetilde{\mathcal{B}}$ is linearly independent.

\medskip

\underline{Step 3:} Pick any $\delta'$ such that $\delta < \delta' < 1$, then apply Proposition \ref{prop:samedegree2} separately in each degree so that each $\Bset_d$ is an asymptotically LLL-reduced basis for the lattice it generates with factor $\delta'$.

We now have a parametric basis $\mathcal{B}$ satisfying the following:
\begin{enumerate}
\setcounter{enumi}{3}
\item The \textbf{eventual} Lov\'asz condition holds on the entire basis $\mathcal{B}$ with factor $\delta$ (by Lemma~\ref{lem:Lovaszbydegree});
\item Each subset $\mathcal{B}_d$ satisfies the \textbf{eventual} size-reduced condition (as a consequence of applying the procedure in Proposition \ref{prop:samedegree2});
\item The set $\widetilde{\mathcal{B}}$ of pilot vectors remain linearly independent (because we applied a unimodular matrix over $\Z$ to each set of vectors $\mathcal{B}_d$); and
\item For $d < e$, $\Bset_d \perp^{asym} \Bset_e$ still holds (by Lemma \ref{lem:orthogonal-span}, again noting that within each degree we applied a unimodular matrix over $\Z$, so that $\Span_\Z(\mathcal{B}_d)$ and $\Span_\Z(\mathcal{B}_e)$ are preserved). 
\end{enumerate}

\medskip

\underline{Step 4:} We are almost done: it only remains to obtain eventual size-reduction between vectors of different degrees without losing any of the conditions (4)-(7) above.


The idea is to reduce each vector by integer multiples of all of the vectors of strictly lower degree, just as in Algorithm \ref{alg:LLL}, but without swapping and without reducing by vectors of the same degree. 

More specifically, for each $d$, 
write $\Bset_d = (\fvec_{d,1},\ldots,\fvec_{d,m_d})$ ordered to be an eventually LLL-reduced basis. Then order all of the vectors as
\[\Bset = \left(\fvec_{(0,1)},\ldots,\fvec_{(0,m_0)},\ldots,\fvec_{(d_{\max},1)},\ldots,\fvec_{(d_{\max},m_{d_{\max}})}\right).\]
We will write $\prec$ for the lexicographic order on pairs of indices; that is, $(d,v) \prec (e,u)$ if $d < e$ or if $d = e$ and $v < u$. 

With this notation, we apply Algorithm \ref{alg:size-tweaking} below. We claim that this finishes the process of LLL reduction:

\begin{algorithm}
  \begin{algorithmic}[1] \caption{Final size reduction between vectors of different degrees} \label{alg:size-tweaking}
  \Statex \textbf{INPUT:} $\Bset = \{ \fvec_{(0,1)},\ldots,\fvec_{(0,m_0)},\ldots,\fvec_{(d_{\max},1)},\ldots,\fvec_{(d_{\max},m_{d_{\max}})}\} \subseteq \Z[t]^m$ 
  \Statex \textbf{OUTPUT:} An LLL-reduced basis $\{ \fvec_{(0,0)},\ldots,\fvec_{(0,m_0)},\ldots,\fvec_{(d_{\max},1)},\ldots,\fvec_{(d_{\max},m_{d_{\max}})}\} \subseteq \textup{EQP}^m$ 
  \State Compute the corresponding Gram-Schmidt vectors $\fvec_{d,u}$ for $0 \leq d \leq d_{\max}$ and $1 \leq u \leq m_d$, and the Gram-Schmidt coefficients $\rho_{(e,u),(d,v)}$ for $(d,v) \prec (e,u)$
  \For {$e = 1$ to $d_{\max}$ and $u=1$ to $m_e$}
   \For{$d = e-1$ down to $0$ and $v=m_d$ down to $1$}
  \State Let $q_{(d,v)} = \lfloor \rho_{(e,u),(d,v)} \rceil \in \textup{EQP}$ and set $\fvec_{(e,u)} = \fvec_{(e,u)} - q_{(d,v)}\fvec_{(d,v)}$
  \State Update the values $\rho_{(e,u),(c,w)}$ for $(c,w) \prec (e,u)$
  \State Branch mod M, where M is the eventual period of the EQP function $q_{(d,v)}$
    \EndFor
    \EndFor
    \end{algorithmic}
\end{algorithm}

\begin{prop}
Algorithm \ref{alg:size-tweaking} terminates and the output basis $\Bset$ satisfies the eventual size-reduced and eventual L\'ovasz conditions (with our chosen value of $\delta$).
\end{prop}

\begin{proof}
Termination of the algorithm is immediate because there are no potentially unbounded loops: although we may branch at each iteration, we immediately advance to the next iteration for every value of $i (mod M)$.

First we check that the output $\mathcal{B}$ satisfies the eventual Lov\'asz condition with factor $\delta$. Recall that in Step 3, we arranged for the eventual Lov\'asz condition to hold between any pair of adjacent vectors, and so we just need to check that the operations performed in Algorithm~\ref{alg:size-tweaking} preserve this condition. This condition automatically holds between vectors of different degrees by Lemma~\ref{lem:lovasz}, so we only need to check it for pairs of adjacent vectors of the same degree. By Proposition~\ref{prop:rho-bounds}, the Gram-Schmidt vectors $\fvec^*_{(d,u)}$ are unchanged by the reduction in line 4 of Algorithm~\ref{alg:size-tweaking}, so it is sufficient to check that the relevant $\rho$ value is also unchanged. On the one hand, suppose that we are considering $\rho_{(e,u), (e, u-1)}$ (for two vectors of degree $e$) and we apply line 4 to replace $\fvec_{(e, u)}$ by $\fvec_{(e, u)} - q_{(d,v)} \fvec_{(d,v)}$. Applying Proposition~\ref{prop:rho-bounds} (3) with $j = (d,v)$, $i = (e, u-1)$, and $k = (e, u)$, we conclude that $\rho_{k,i} = \rho_{(e,u), (e, u-1)}$ is unchanged, as desired. The only other case to consider is when we apply line 4 to replace $\fvec_{(e, u-1)}$ by $\fvec_{(e, u-1)} - q_{(d,v)} \fvec_{(d,v)}$, but in this case $$\rho_{(e,u), (e, u-1)} = \frac{\langle \fvec_{(e,u)}, \fvec^*_{(e,u-1)} \rangle}{\langle \fvec^*_{(e, u-1)}, \fvec^*_{(e,u-1)} \rangle}$$ is also unaffected since by Proposition~\ref{prop:rho-bounds} again the Gram-Schmidt vector $\fvec^*_{(e, u-1)}$ is not changed.

Finally, we verify that the output $\mathcal{B}$ satisfies the eventual size-reduced condition.

  First consider two pairs of indices $(e,u), (d,v)$ with $d < e$. At one point during Step 4 we will reduce $\fvec_{e,u}$ by $\fvec_{d,v}$ and obtain $|\rho_{(e,u),(d,v)}(t) | \leq 1/2$ for sufficiently large $t$ by Proposition \ref{prop:rho-bounds} (2).  Proposition \ref{prop:rho-bounds} (3) guarantees that $\rho_{(e,u),(d,v)}$ does not change again as we reduce $\fvec_{e,u}$ by other vectors earlier in the list, so we have achieved eventual size-reduction between vectors of different degrees.

  Now consider pairs of vectors of the same degree. For any degree $e$ and for $1 \leq v < u \leq m_e$ we already had $|\rho_{(e,u),(e,v)}|$ eventually less than or equal to 1/2 before starting Step 4. Since we only reduce $\fvec_{e,u}$ by vectors of degree less than $e$, which are earlier in the list than $\fvec_{e,v}$, $\rho_{(e,u),(e,v)}$ remains the same throughout Step 4 by Proposition \ref{prop:rho-bounds} (3). Thus eventual size-reduction continues to hold. 
  
 \end{proof}

\begin{proof}[Proof of Theorem \ref{thm:LLLreduced}:]
Apply Steps 1 through 4 above.
\end{proof}

\section{Shortest vector and closest vector problems}

\subsection{Shortest vector problem}

We now prove Theorem \ref{thm:shortest}. The only additional ingredient is a standard fact about comparison of polynomials and its extension to EQP functions, as follows. 

\begin{prop} \label{prop:polycomparison} 
Let $p_1, \dots, p_N \in \Z[t]$ be polynomials. There exists $i$ such that for all sufficiently large $t$, $p_i(t) \geq p_k(t)$ for every $k = 1,\dots,N$.  
\end{prop} 

\begin{cor} \label{cor:EQPcomparison}
Let $q_1, \dots, q_N$ be EQP functions. Then there exists a modulus $M$ and indices $i_1, \dots, i_M$ such that for all $j=1,\dots,M$ and all sufficiently large $t$ congruent to $j$ (mod $M$), $q_{i_j}(t) \geq q_k(t)$ for every $k=1,\dots,N$.  
\end{cor}

\begin{proof}
Choose $T$ such that $q_1, \dots, q_N$ each agree with a quasi-polynomial for $t > T$ and let $M$ be a common period of the $N$ quasi-polynomials. For $t > T$ and congruent to a fixed $j$ (mod M), we thus compare only polynomials, so the result follows from Proposition \ref{prop:polycomparison}.
\end{proof}

\begin{proof}[Proof of Thoerem \ref{thm:shortest}]
Let $\Lambda_t = \Span_\Z \{\fvec_1(t), \dots, \fvec_n(t)\} \subseteq \Z[t]^m$. By Theorem \ref{thm:LLLreduced}, we may first obtain $\gvec_1(t),\ldots, \gvec_{n'}(t)$ with EQP coordinates, which eventually form an LLL-reduced basis of $\Lambda_t$, using the traditional $\delta=\frac{3}{4}$ for the Lov\'asz condition. For notational convenience we write $n$ instead of $n'$ in the arguments below.

To prove Theorem \ref{thm:shortest}, we want to describe a shortest nonzero vector in $\Lambda_t$, using EQP coordinates. For a given $t$, let $\uvec(t) \in\Lambda_t$ be \emph{any} shortest nonzero vector in $\Lambda_t$. By Fact~\ref{fact:coefbounds} above,
$\uvec(t) = \sum_i a_i\gvec_i(t)$, with $a_i\in\Z$ and $\abs{a_i}\le 3^n$. In other words, every shortest nonzero vector in $\Lambda_t$ is one of the $N := n^{2\cdot 3^n +1}$ vectors $\sum_i a_i\gvec_i(t)$ with $a_i\in\Z$ and $\abs{a_i}\le 3^n$. Note that this is a \text{fixed} number of vectors, independent of $t$. Since $\norm{\sum_i a_i\gvec_i(t)}^2$ are EQP functions, one of them will eventually be minimal within each residue class by Corollary \ref{cor:EQPcomparison}.

That is, as a function of $t$ the shortest vector is of the form $\uvec(t) = \sum_i a_i\gvec_i(t)$ where each $\gvec_i$ has EQP coordinates and the choice of $i$ is eventually periodic, so $\uvec(t)$ also has EQP coordinates.  
\end{proof}

\subsection{Closest vector problem}

To prove Theorem \ref{thm:closest}, we want to describe a vector in $\Lambda_t$ closest to $\xvec(t)\in \Q(t)^m$, using EQP coordinates. First, we show that we can reduce to the case where $\xvec(t)$ is in the subspace $\Span_\R \Lambda_t$.

\begin{lem}
\label{lem:orthog}
Let $\Lambda_t = \Span_\Z \{\fvec_1(t), \dots, \fvec_n(t)\} \subseteq \Z[t]^m$ and let $V_t=\Span_\R \Lambda_t$. Let $\xvec(t)\in \Q(t)^m$ and let $\yvec(t)$ be the orthogonal projection of $\xvec(t)$ onto $V_t$. Let $\uvec\in\Lambda_t$. Then
\begin{enumerate}
\item $\yvec$ is in $\Q(t)^m$ and
\item $\uvec$ is a closest vector (in $\Lambda_t$) to $\xvec(t)$ iff $\uvec$ is a closest vector to $\yvec(t)$.
\end{enumerate}

\begin{proof}
Let $A_t$ be the matrix whose columns are $\gvec_1(t),\ldots,\gvec_n(t)$. Then
\[\yvec(t) = \left(A_t^TA_t\right)^{-1}A_t^T\xvec(t).\]
Noting that the inverse of a matrix can be written using adjoints (and dividing by the determinant), this result is in $\Q(t)^m$, proving (1). Next, since $(\yvec(t)-\uvec)\in V_t$ and $(\xvec(t)-\yvec(t))$ are orthogonal, we have
\[\norm{\xvec(t)-\uvec}^2=\norm{\yvec(t)-\uvec}^2+\norm{\xvec(t)-\yvec(t)}^2,\]
and so $\uvec$  minimizes $\norm{\xvec(t)-\uvec}$ if and only if it minimizes $\norm{\yvec(t)-\uvec}$, proving (2).
\end{proof}
\end{lem}

Now we proceed by induction on $n$, which is $\dim(\Lambda_t)$ (for sufficiently large $t$). When $n=0$, $\Lambda_t=\{0\}$ and so $\uvec(t)=0$ is the closest lattice vector to $\xvec(t)$. Assume the theorem is true for an $n-1$ dimensional lattice, and let $\gvec_1(t), \dots, \gvec_n(t)$ (with coordinates in EQP) be an eventually LLL-reduced basis (with  $\delta=\frac{3}{4}$) for our $n$-dimensional lattice, $\Lambda_t$. Using Lemma \ref{lem:orthog}, we may assume without loss of generality that $\xvec(t)\in V_t=\Span_\R \Lambda_t$.

Let $\gvec^*_1,\ldots, \gvec^*_n$ be the Gram-Schmidt vectors corresponding to $\gvec_1,\ldots, \gvec_n$. For a given $t$, let $\uvec\in\Lambda_t$ be a vector closest to $\xvec(t)$. Write
\[\uvec=\sum_{i=1}^na_i\gvec_i(t)=\sum_{i=1}^n\mu_i\gvec^*_i(t)\quad \text{and} \quad \xvec(t)=\sum_{i=1}^nc_i\gvec_i(t)=\sum_{i=1}^n\nu_i\gvec^*_i(t),\]
with $a_i\in\Z$ and $c_i,\mu_i,\nu_i\in\Q$ (for $1\le i\le n$), and note $a_n=\mu_n$ and $c_n=\nu_n$. We first seek a bound (independent of $t$) on $\abs{a_n-c_n}$, which will allow us to inductively search a fixed number of smaller dimensional lattice translates for the closest point.

To find this bound, we first use Babai's nearest plane method (see Fact~\ref{fact:babai} above), which gives a reasonably close lattice point, $\wvec\in\Lambda_t$, with the property that
\[\label{eqn:babai}\norm{\wvec-\xvec}\le 2^{n/2 - 1}\norm{\gvec^*_n}.\]

Now we have
\begin{align*}
\abs{\mu_n-\nu_n}\cdot\norm{\gvec^*_n}&\le \sqrt{\sum_{i=1}^n(\mu_i-\nu_i)^2\cdot\norm{\gvec^*_i}^2}\\
&= \Big|\Big|\sum_{i=1}^n(\mu_i-\nu_i)\gvec^*_i\Big|\Big|\\
&= \norm{\uvec-\xvec}\\
&\le \norm{\wvec-\xvec}\\
&\le 2^{n/2 - 1}\norm{\gvec^*_n}.
\end{align*}
Since $a_n=\mu_n$ and $c_n=\nu_n$, dividing both sides by $\norm{\gvec^*_n}$ yields our desired bounds:
\begin{equation}\label{eqn:cvp_bound}\abs{a_n-c_n}\le 2^{n/2 - 1}.\end{equation}

Let $\Lambda'_t=\Span_\Z \{\gvec_1(t), \dots, \gvec_{n-1}(t)\} \subseteq \Z[t]^m$, so that $\Lambda_t=\Lambda'_t+\Z\gvec_n$. Note that $c_n=c_n(t)$ is \emph{eventually quasi-rational} (meaning that there is some $N$ such that for large enough $t$, $c_n(t)$ can be expressed by one of $N$ different quotients of polynomials over $\Q$, depending on the congruence class of $t$ modulo $N$).  This implies that $\floor{c_n(t)}$ is EQP. Let $I=\left\{-2^{n/2 - 1},-2^{n/2 - 1}+1,\ldots,-1,0,1,\ldots,2^{n/2 - 1}+1\right\}$. The bounds (\ref{eqn:cvp_bound}) imply that, if $\uvec$ is a closest vector (in $\Lambda_t$) to $\xvec(t)$, then $\uvec\in \Lambda'_t+a_n\gvec_n$ for $a_n\in I+\floor{c_n(t)}$. This yields an inductive algorithm to find a closest vector with EQP coordinates: For $i
\in I$, let $\vvec^{(i)}(t)$ be a closest vector to $\xvec(t)-\left(i+\floor{c_n(t)}\right)\gvec_n$ in $\Lambda'$. By the inductive hypothesis, we can take $\vvec^{(i)}(t)$ with EQP coordinates. Then $\vvec^{(i)}(t)+\left(i+\floor{c_n(t)}\right)\gvec_n$ is a closest vector to $\xvec(t)$ in $\Lambda'_t+\left(i+\floor{c_n(t)}\right)\gvec_n$. The closest vector in $\Lambda$ must be one of these $2^{n/2}+2$ vectors $\vvec^{(i)}(t)+\left(i+\floor{c_n(t)}\right)\gvec_n$. We can compare these vectors with EQP coordinates, and one of them must eventually minimize $\norm{\vvec^{(i)}(t)+\left(i+\floor{c_n(t)}\right)\gvec_n - \xvec}$ and be a closest vector, for sufficiently large $t$.

\subsection{Acknowledgements} The first author was supported by an internal research grant (INV-2017-51-1453) from the Faculty of Sciences of the Universidad de los Andes during his work on this project. The first author would also like to thank the Oberlin College Mathematics Department for inviting him to be the 2018 Lenora Young Lecturer and facilitating his collaboration with the third author.  
\bibliography{LLL}
\end{document}